\documentclass[12pt]{amsart}
\usepackage{amssymb}
\usepackage{hyperref}

\theoremstyle{plain}
\newtheorem{theorem}{Theorem}[section]
\newtheorem{proposition}[theorem]{Proposition}
\newtheorem{corollary}[theorem]{Corollary}
\newtheorem{lemma}[theorem]{Lemma}

\theoremstyle{definition}
\newtheorem{definition}[theorem]{Definition}
\newtheorem{example}[theorem]{Example}
\newtheorem{remark}[theorem]{Remark}

\newtheorem{conjecture}[theorem]{Conjecture}

\theoremstyle{remark}
\newtheorem{notation}[theorem]{Notation}

\numberwithin{equation}{section}

\newcommand{\N}{\mathbb N}
\newcommand{\Z}{\mathbb Z}

\newcommand{\R}{\mathbb R}
\newcommand{\C}{\mathbb C}
\newcommand{\Ha}{\mathbb H}

\newcommand{\GL}{\operatorname{GL}}
\newcommand{\SL}{\operatorname{SL}}

\newcommand{\SO}{\operatorname{SO}}
\newcommand{\SU}{\operatorname{SU}}
\newcommand{\U}{\operatorname{U}}
\newcommand{\Sp}{\operatorname{Sp}}

\newcommand{\Spin}{\operatorname{Spin}}

\newcommand{\sll}{\mathfrak{sl}}
\newcommand{\so}{\mathfrak{so}}

\newcommand{\su}{\mathfrak{su}}

\newcommand{\op}{\operatorname}

\newcommand{\Spec}{\operatorname{Spec}}
\newcommand{\Hom}{\operatorname{Hom}}
\newcommand{\Iso}{\operatorname{Iso}}

\newcommand{\diag}{\operatorname{diag}}

\newcommand{\ba}{\backslash}
\newcommand{\mult}{\operatorname{mult}}
\newcommand{\mi}{\mathtt{i}}
\newcommand{\norma}[1]{\|{#1}\|}

\newcommand{\Cas}{\textit{Cas}}

\title[Spectra of orbifold]{Spectra of orbifolds with cyclic fundamental groups}
\author{Emilio A. Lauret}
\address{CIEM--FaMAF \\ Universidad Nacional de C\'ordoba\\ 5000-C\'ordoba, Argentina.}
\email{elauret@famaf.unc.edu.ar}
\date{\today}

\begin{document}

\begin{abstract}
We give a simple geometric characterization of isospectral orbifolds covered by spheres, complex projective spaces and the quaternion projective line having cyclic fundamental group.
The differential operators considered are Laplace-Beltrami operators twisted by characters of the corresponding fundamental group.
To prove the characterization, we first give an explicit description of their spectra by using generating functions.
We also include many isospectral examples.
\end{abstract}

\maketitle

\section{Introduction}
Let $M$ be a connected compact Riemannian manifold, and let $\Gamma_1$ and $\Gamma_2$ be cyclic subgroups of $\Iso(M)$.
In this paper we will consider the question:
\begin{quote}
\emph{Under what conditions the (good) orbifolds $\Gamma_1\ba M$ and $\Gamma_2\ba M$ are isospectral?}
\end{quote}
Here, isospectral means that the Laplace-Beltrami operators on $\Gamma_1\ba M$ and on $\Gamma_2\ba M$ have the same spectrum.
Such operators are given by the Laplace-Beltrami operator on $M$ acting on $\Gamma_j$-invariant smooth functions on $M$.

Clearly, the condition on $\Gamma_1$ and $\Gamma_2$ of being conjugate in $\Iso(M)$ is sufficient since in this case $\Gamma_1\ba M$ and $\Gamma_2\ba M$ are isometric.
However, it is well known that this condition is not necessary due to examples of non-isometric isospectral lens spaces constructed by Ikeda~\cite{Ikeda80_isosp-lens}.
Lens spaces are spherical space forms with cyclic fundamental groups, that is, they are the only compact manifolds with constant sectional curvature and cyclic fundamental group.

In the recent paper \cite{LMR15-onenorm}, R.~Miatello, J.P.~Rossetti and the author show an isospectral characterization among lens spaces in terms of isospectrality of their associated congruence lattices with respect to the \emph{one-norm} $\norma{\cdot}_1$.
More precisely, each cyclic subgroup $\Gamma$ of $G=\SO(2n)=\op{Iso}^+(S^{2n-1})$ has associated a sublattice $\mathcal L_\Gamma$ of the weight lattice of $G$.
Then, it is shown that $\Gamma_1\ba S^{2n-1}$ and $\Gamma_2\ba S^{2n-1}$ are isospectral if and only if for each $k\geq0$ there are the same number of elements in $\mathcal L_{\Gamma_1}$ and $\mathcal L_{\Gamma_2}$ having one-norm equal to $k$.

The aim of this paper is to extend the results in \cite{LMR15-onenorm} to more general locally homogeneous compact Riemannian manifolds having cyclic fundamental group.
Indeed, we will focus here on spaces covered by compact symmetric spaces of real rank one.

\subsection{Method}

We fix an arbitrary compact connected Riemannian manifold $(M,g)$.
The spectrum of the Laplace-Beltrami operator $\Delta$ acting on smooth functions on $M$ is given by a sequence of eigenvalues
\begin{equation}
0=\lambda_0<\lambda_1<\lambda_2<\dots,
\end{equation}
each of them with finite multiplicity
\begin{equation}
\mult_{\lambda_k}(\Delta):=\dim \op{Eig}_{\lambda_k}(\Delta).
\end{equation}
Here, $\op{Eig}_{\lambda_k}(\Delta)$ denotes the eigenspace of $\Delta$ with eigenvalue $\lambda_k$.
The \emph{spectral zeta function} of $\Delta$, also called \emph{Minakshisundaram-Pleijel zeta function}, is given by
\begin{equation}
\zeta_{\Delta}(s) = \sum_{k\geq1} \mult_{\lambda_k}(\Delta)\; \lambda_k^{-s}.
\end{equation}

Let $\Gamma$ be a discrete cocompact subgroup of $\Iso(M)$.
The space $\Gamma\ba M$ is compact and inherits a structure of good Riemannian orbifold $(\Gamma\ba M,g)$, which is a manifold if and only if $\Gamma$ acts freely on $M$.
The operator $\Delta$ commutes with isometries, thus the Laplace-Beltrami operator $\Delta_{\Gamma}$ on $(\Gamma\ba M,g)$ coincides with $\Delta$ restricted to $\Gamma$-invariant smooth functions on $M$.
Consequently, any eigenvalue in the spectrum of $\Delta_{\Gamma}$ is in the set $\{\lambda_k:k\geq0\}$, and
\begin{equation}
  \mult_{\lambda_k}(\Delta_{\Gamma}) = \dim \op{Eig}_{\lambda_k}(\Delta)^\Gamma.
\end{equation}
Hence, $0\leq \mult_{\lambda_k}(\Delta_{\Gamma})\leq \mult_{\lambda_k}(\Delta)$
and $\lambda_k$ is an eigenvalue of $\Delta_{\Gamma}$ unless $\dim \op{Eig}_{\lambda_k}(\Delta)^\Gamma=0$.

Following an idea of Ikeda, we encode the spectrum of $\Delta_{\Gamma}$ via the following formal power series
\begin{equation}\label{eq1:F(z)}
F_{\Delta_\Gamma}(z) := \sum_{k\geq0} \mult_{\lambda_k}(\Delta_\Gamma)\, z^k,
\end{equation}
that we called the \emph{spectral generating function} associated to $\Gamma\ba M$.
Clearly, $\Gamma_1\ba M$ and $\Gamma_2\ba M$ are isospectral if and only if $F_{\Delta_{\Gamma_1}}(z) = F_{\Delta_{\Gamma_2}}(z)$.

We now assume that $M=G/K$ is a compact homogeneous Riemannian manifold and $\Gamma$ is a cyclic subgroup of $G$ contained in a (fixed) maximal torus $T$.
These hypotheses make easier the explicit computation of $\dim \op{Eig}_{\lambda_k}(\Delta)^\Gamma$.
Indeed, $\op{Eig}_{\lambda_k}(\Delta)$ is a finite dimensional representation of $G$, so it decomposes diagonally into weight spaces with respect to $T$.
Hence, $\op{Eig}_{\lambda_k}(\Delta)$ is the sum over the set $\mathcal L_\Gamma$ of $\Gamma$-invariant weights of the corresponding weight spaces.
Consequently, its dimension is equal to the sum over elements in $\mathcal L_\Gamma$ of the corresponding multiplicities.
In conclusion, the determination of $\mult_{\lambda_k}(\Delta_{\Gamma})$ reduces to computing the multiplicities of the elements in $\mathcal L_\Gamma$ in the representation $\op{Eig}_{\lambda_k}(\Delta)$ of the group $G$.

\subsection{Results}
For compact symmetric spaces of real rank one, the representations $\op{Eig}_{\lambda_k}(\Delta)$ for $k\geq0$ are spherical representations, which are well known.
Their multiplicities satisfy a simple (geometric) formula when $M=P^n(\C)$, $S^{2n}$, $P^1(\Ha)$ or $S^{2n-1}$ for $n\geq1$ (Proposition~\ref{prop:multiplicities}).
More precisely, in the weight lattice associated to $G=\SU(n+1)$, $\SO(2n+1)$, $\Sp(2)$ or $\SO(2n)$, there is a norm $\norma{\cdot}$ such that the multiplicity of a weight depends only on its norm value.
Such norm is the \emph{maximum norm} $\norma{\cdot}_\infty$ for $\Sp(2)$ and the \emph{one-norm} $\norma{\cdot}_1$ for the other cases.

Let $\Gamma$ be a subgroup of a maximal torus $T$ of $G$.
Let $\mathcal L_\Gamma$ be the lattice given by $\Gamma$-invariant weights.
The mentioned geometric condition satisfied by the multiplicities of the spherical representations, allows us to give explicit formulas for $\mult_{\lambda_k}(\Delta_{\Gamma})$ in terms of the number of weights in $\mathcal L_\Gamma$ with a fixed norm (Theorem~\ref{thm:multiplicities}).
More precisely, in terms of $N_{\mathcal L_\Gamma}(k):=\#\{\mu\in\mathcal L_\Gamma:\norma{\mu}=k\}$.
We define in \eqref{eq3:vartheta_L} the \emph{generating theta function} associated to the lattice $\mathcal L_\Gamma$ given by \begin{equation}\label{eq1:theta(z)}
  \vartheta_{\mathcal L_{\Gamma}}(z) = \sum_{k\geq0} N_{\mathcal L_\Gamma}(k) \, z^k.
\end{equation}

In Theorem~\ref{thm:spectradescription} we prove
\begin{equation}\label{eq1:Fformula}
F_{\Delta_\Gamma} (z)=
\begin{cases}
\dfrac{1}{(1-z)^{n}}\; \vartheta_{\mathcal L_{\Gamma}}(z)
    &\quad\text{if $M=P^n(\C)$,}\\[5mm]
\dfrac{1+z}{(1-z^2)^{n}}\; \vartheta_{\mathcal L_{\Gamma}}(z)
    &\quad\text{if $M=S^{2n}$,}\\[5mm]
\dfrac{1+z}{(1-z^2)^{2}}\; \vartheta_{D_2\cap\mathcal L_{\Gamma}}(z)
    &\quad\text{if $M=P^1(\Ha)$,}\\[5mm]
\dfrac{1}{(1-z^2)^{n-1}}\; \vartheta_{\mathcal L_{\Gamma}}(z)
    &\quad\text{if $M=S^{2n-1}$.}
\end{cases}
\end{equation}
Here, $D_2=\{(a,b)\in\Z^2: a+b\equiv0\pmod2\}$.
This formula is an alternative and simple way to describe the spectrum of the Laplace-Beltrami operator $\Delta_{\Gamma}$ on $\Gamma\ba M$, for $M=P^n(\C)$, $S^{2n}$, $P^1(\Ha)$ and $S^{2n-1}$ for $\Gamma$ included in $T$.

Clearly, \eqref{eq1:Fformula} gives an isospectral characterization among locally symmetric spaces of the form $\Gamma\ba M$ with $M$ as above.
Namely, $\Gamma\ba M$ and $\Gamma'\ba M$ are isospectral if and only if $\vartheta_{\mathcal L_{\Gamma}}(z)=\vartheta_{\mathcal L_{\Gamma'}}(z)$ (Theorem~\ref{thm:characterization}).

We actually consider in all previous results, Laplace operators twisted by a character of the fundamental group.
The spectrum of such operator has a similar description as in the previous case (trivial character).
Its associated set $\mathcal L_{\Gamma,\chi}$ is in this case a shifted lattice of $\mathcal L_\Gamma$ and formula \eqref{eq1:Fformula} is still valid.

Furthermore, in the untwisted case (i.e.\ $\chi$ trivial), we write the generating theta function $\vartheta_{\mathcal L_\Gamma}(z)$ as a rational function, by using Ehrhart's theory for counting integer points in polytopes (Theorem~\ref{thm:Ehrhart}).
This implies that $F_{\Delta_\Gamma}(z)$ is a rational function.
We also give in Proposition~\ref{prop:Zagier} an alternative description of $\vartheta_{\mathcal L_\Gamma}(z)$ for lens spaces $\Gamma\ba S^{2n-1}$, which implies a formula for $F_{\Delta_\Gamma}(z)$ already known by Ikeda (see Remark~\ref{rem:Ikeda-formula}).

\subsection{Examples}
In Section~\ref{sec:isospectrality} we find twisted and untwisted isospectral examples by making use of the isospectral characterization in Theorem~\ref{thm:characterization} and the explicit parametrization of cyclic subgroups in $G$ (see \S\ref{subsec:classicalgroups}, \S\ref{subsec:cyclic}), where $M=G/K$ is one of the cases considered.
By using the computer, we found for small values of $n$ and $q$, every pair of twisted and untwisted isospectral orbifolds covered by $P^n(\C)$, $S^{2n}$ and $S^{2n-1}$, and also by $P^1(\Ha)$ for $n=2$, having cyclic fundamental groups of order $q$.

The most significant examples can be summarized in the next observations:
\begin{enumerate}
\item there exist twisted isospectral lens spaces in dimension 3 (Remark~\ref{rem:supportconjecture});
\item there exist no untwisted isospectral orbifolds covered by $S^3$ and $S^4$ with cyclic fundamental group of order $q\leq 200$ (Remark~\ref{rem:supportconjecture});
\item there exist untwisted isospectral orbifold lens spaces covered by $S^d$ with $5\leq d\leq 8$, thus such examples exists for any dimension $d\geq 5$ by \cite{Shams11} (Remark~\ref{rem:Shams});
\item there exists a $3$-dimensional lens space (a manifold) that is twisted isospectral to an orbifold lens space with non-trivial singularities (Remark~\ref{rem:twistediso-mfd-orb});
\item there are examples of twisted and untwisted isospectrality covered by $P^n(\C)$ for $n=2,3$, that is, in dimension $4$ and $6$;
\item there exists a pair of untwisted isospectral orbifolds covered by $P^1(\Ha)$ with cyclic fundamental group of order $q=4$, the one with minimum dimension (Example~\ref{ex:P^1(H)q=4}).
\end{enumerate}

\subsection{Previous results}

Lens spaces have been a fertile ground for spectral inverse problems.
Ikeda and Yamamoto~\cite{IkedaYamamoto79}\cite{Yamamoto80} proved that there are no isospectral lens spaces in dimension $3$ by giving an explicit formula for the spectral generating function $F_{\Delta_\Gamma}(z)$ by using Molien's formula.
Boldt~\cite{Boldt15} proved that there are no Dirac isospectral $3$-dimensional lens spaces with fundamental group of order prime. 
His proof uses $p$-adic numbers of cyclotomic fields and the explicit formula for the spectral generating functions associated to the positive and negative eigenvalues of the Dirac operator given by B\"ar~\cite{Bar96}.
Furthermore, Ikeda used the mentioned explicit formula for $F_{\Delta_\Gamma}(z)$ to construct several examples (from dimension $5$ on) of isospectral lens spaces (\cite{Ikeda80_isosp-lens}).
Generalizing such formula for the Hodge-Laplace operators on $p$-forms, he also obtained for each $p_0>0$, examples of lens spaces that are $p$-isospectral for every $0\leq p<p_0$ but not $p_0$-isospectral (\cite{Ikeda88}).
Shams~\cite{Shams11} generalizes Ikeda's method to orbifold lens spaces, obtaining examples from any dimension greater than $8$.

As we mentioned above, in \cite{LMR15-onenorm}, Miatello, Rossetti and the author prove the characterization in Theorem~\ref{thm:characterization} for lens spaces.
Furthermore, it shows a similar characterization for lens spaces $p$-isospectral for all $p$, which allows them to construct the first pair of Riemannian manifolds $p$-isospectral for all $p$ but not strongly isospectral.
DeFord and Doyle~\cite{DeFordDoyle14}, answering a question in \cite{LMR15-onenorm}, obtain a sufficient condition to construct such examples.
In \cite{BoldtLauret-onenormDirac}, Boldt and the author generalize the method in \cite{LMR15-onenorm} to the Dirac operator on spin lens spaces.
The article \cite{LMR-survey} contains a summary of all these results.

Similarly as in Theorem~\ref{thm:Ehrhart}, Mohades and Honari~\cite{MohadesHonari16} have recently used Ehrhart's theory for counting points in rational polytopes to describe the spectrum of a lens space.
They also associate to each lens space a toric variety, obtaining geometric consequences for pairs of untwisted isospectral lens spaces.

\subsection{Acknowledgement}
The author wishes to thank Leandro Cagliero and Jorge Vargas for helpful conversations on representation theory and also Sebastian Boldt and Ramiro Lafuente for helpful comments concerning the algorithms used in the last section.
The author also wishes to thank the support of the Oberwolfach Leibniz Fellow programme in May-July 2013 and in August-November 2014, when this project started.

\section{Preliminaries}\label{sec:abelian-case}
This section reviews some of the standard facts on the spectra of twisted Laplace operators on locally symmetric spaces of compact type with cyclic fundamental groups.

\subsection{Spectra of twisted Laplace operators}\label{subsec:spectra}

Let $G$ be a Lie group and let $(M,g)$ be a connected compact $G$-homogeneous Riemannian manifold, that is, $G$ acts transitively and almost effectively by isometries on $M$.
Let us denote by $K$ the isotropy subgroup of $G$ at some point $p\in M$, and let $\mathfrak g$ and $\mathfrak k$ be the Lie algebras of $G$ and $K$ respectively.
It turns out that $K$ is compact and we have the identifications $M\simeq G/K$, $g\cdot p\leftrightarrow gK$, and $T_pM\simeq T_{eK}G/K\simeq \mathfrak p$, where $\mathfrak p$ is an $\op{Ad}(K)$-invariant subspace of $\mathfrak g$ satisfying $\mathfrak g=\mathfrak k\oplus \mathfrak p$.
It is well known that the $G$-invariant Riemannian metrics on $M$ are in correspondence with the $\op{Ad}(K)$-invariant inner products on $\mathfrak p$.

We will restrict our attention to the standard metric, that is, the Riemannian metric is given by a negative multiple of the Killing form $B(\cdot,\cdot)$ restricted to $\mathfrak p\times\mathfrak p$.
Note that in this case $G$ must be compact and semisimple.
However, we will work in this section under the slightly weaker assumption described in the next paragraph.

We assume that $G$ is a compact Lie group (not necessarily semisimple).
Let $\mathfrak g$ denote its Lie algebra, thus $\mathfrak g=\mathfrak z\oplus \mathfrak g_1$ with $\mathfrak z$ the center of $\mathfrak g$ and $\mathfrak g_1=[\mathfrak g,\mathfrak g]$ the semisimple part.
Let $\langle\cdot,\cdot\rangle$ be an inner product on $\mathfrak g$ satisfying
\begin{itemize}
  \item $\langle \mathfrak z,\mathfrak g_1\rangle=0$;
  \item $\langle\cdot,\cdot\rangle|_{\mathfrak g_1\times \mathfrak g_1} = -rB_{\mathfrak g_1}$ for some $r>0$;
  \item there is an orthonormal free basis of $\{Z\in\mathfrak z: \exp Z=e\}$.
\end{itemize}
We assume that the Riemannian metric $g$ on $M$ corresponds to $\langle\cdot,\cdot\rangle|_{\mathfrak p\times\mathfrak p}$.

Clearly, every discrete subgroup $\Gamma$ in  $G$ is finite.
The space $\Gamma\ba M$ inherits a structure of Riemannian (good) orbifold, which is a Riemannian manifold when $\Gamma$ acts freely on $M$ (see for instance \cite{Gordon12-orbifold}).

Let $\Gamma$ be a finite subgroup of $G$ and let $\chi:\Gamma\to\GL(W_\chi)$ be a finite dimensional representation of $\Gamma$.
We consider the bundle
$$
E_{\Gamma,\chi}:= \Gamma\ba G\times_\chi W_\chi \to \Gamma\ba G/K
$$
given by the direct product $G\times W_\chi$ under the relation $(g,w)\sim(\gamma gk,\chi(\gamma)w)$ for every $k\in K$ and $\gamma\in\Gamma$.
We will denote by $[x,w]$ the class of $(x,w)\in G\times W_\chi$ in $E_{\Gamma,\chi}$.
The space $\Gamma^\infty(E_{\Gamma,\chi})$ of smooth sections of $E_{\Gamma,\chi}$ is isomorphic to $C^\infty(\Gamma\ba G/K;\chi):=\{f:G\to W_\chi\,\text{ smooth}: f(\gamma xk)=\chi(\gamma)f(x)\;\forall\,\gamma\in\Gamma,\,k\in K\}$.
Indeed, the function $xK\mapsto[x,f(x)]$ is in $\Gamma^\infty(E_{\Gamma,\chi})$ if and only if $f\in C^\infty(\Gamma\ba G/K;\chi)$.

The Lie algebra $\mathfrak g$ of $G$ acts on $C^\infty(\Gamma\ba G/K,\chi)$ by
$$
(Y\cdot f)(x) =\left.\frac{d}{dt}\right|_{t=0} f(\exp(tY)x).
$$
This action induces a representation of the universal enveloping algebra $U(\mathfrak g)$ of $\mathfrak g$.
The \emph{Casimir element} $\Cas\in U(\mathfrak g)$ is given by $\Cas=\sum_i X_i^2\in U(\mathfrak g)$ where $X_1,\dots,X_n$ is any orthonormal basis of $\mathfrak g$.
The element $-\Cas$ induces a self-adjoint elliptic differential operator $\Delta_{M,\Gamma,\chi}$ of second degree on $C^\infty(\Gamma\ba G/K;\chi)$ and therefore on $\Gamma^\infty(E_{\Gamma,\chi})$.

\begin{definition}
The operator $\Delta_{M,\Gamma,\chi}$ acting on the smooth section on $E_{\Gamma,\chi}$ is called the \emph{$\chi$-twisted Laplace operator}.
When $\chi=1_\Gamma$ (the trivial representation of $\Gamma$), $\Delta_{M,\Gamma}:=\Delta_{M,\Gamma,\chi}$ is called the \emph{untwisted Laplace operator} or just the \emph{Laplace operator}.
\end{definition}

When no confusion can arise on $M$, we just write $\Delta_{\Gamma,\chi}$ and $\Delta_{\Gamma}$ in place of $\Delta_{M,\Gamma,\chi}$ and $\Delta_{M,\Gamma}$ respectively.

Our goal in this subsection is to describe $\Spec(\Gamma\ba M,\chi)$, the spectrum of $\Delta_{\Gamma,\chi}$.
This is given by the multiset of eigenvalues $0<\lambda_1<\lambda_2<\dots$ of $\Delta_{\Gamma,\chi}$ with their corresponding multiplicities $\mult_{\Delta_{\Gamma,\chi}}(\lambda_k)$ for $k\geq1$.

Let $\mathfrak t_\C$ be a Cartan subalgebra of $\mathfrak g_\C$ and fix a positive system $\Sigma^+(\mathfrak g_\C,\mathfrak t_\C)$ in the corresponding root system.
Let us denote by $\widehat G$ the unitary dual of $G$, that is, the class of (finite dimensional) unitary irreducible representations of $G$.
It is well known that the Casimir element $\Cas$ acts by a scalar on each $\pi\in\widehat G$ (see for instance \cite[Lemma~5.6.4]{Wallach-book}).
We denote by $\lambda(\Cas,\pi)$ the opposite of this scalar, so
\begin{equation}\label{eq2:lambda(C,pi)}
  \lambda(\Cas,\pi):= \langle\Lambda_\pi+\rho,\Lambda_\pi+\rho\rangle-\langle\rho,\rho\rangle,
\end{equation}
where $\Lambda_\pi$ is the highest weight of $\pi$ and $\rho$ denotes half the sum of the positive roots relative to $\Sigma^+(\mathfrak g_\C,\mathfrak t_\C)$.

\begin{notation}
For an arbitrary closed subgroup group $H$ of $G$ and $\tau_1,\tau_2$ representations of $H$, set $[\tau_1:\tau_2]=\dim \Hom_H(\tau_1,\tau_2)$.
For example: if $\tau_1$ is irreducible, then $[\tau_1:\tau_2]$ is the number of times that $\tau_1$ appears in the decomposition into irreducibles of $\tau_2$;
if $1_H$ denotes the trivial representation of $H$ and $\pi$ is a representation of $G$, then $[1_H:\pi|_H] = \dim V_\pi^H$, the dimension of the subspace of $H$-invariants in $V_\pi$.
For $\lambda\in\R$, we set
\begin{align*}
\widehat G_{H} &= \{\pi\in\widehat G: [1_H:\pi|_H]>0\},\\
\widehat G(\lambda) &= \{\pi\in\widehat G: \lambda(\Cas,\pi)=\lambda\},\\
\widehat G_{H}(\lambda) &= \widehat G_{H}\cap \widehat G(\lambda).
\end{align*}
The elements in $\widehat G_H$ are usually called the $H$-spherical representations of $G$.
\end{notation}

The next result describes the spectrum of $\Delta_{\Gamma,\chi}$ in terms of $\chi$-invariants of $K$-spherical representations of $G$ restricted to $\Gamma$.

\begin{theorem}\label{thm:spec_chi_gral}
Let $M=G/K$ be a compact homogeneous manifold with a Riemannian metric as above, let $\Gamma$ be a finite subgroup of $G$ and let $\chi$ be a finite dimensional representation of $\Gamma$.
The set of eigenvalues in $\Spec(\Gamma\ba M,\chi)$ is contained in
\begin{equation}\label{eq2:eigenvalues_gral}
  \mathcal E(M):=\{\lambda(\Cas,\pi): \pi\in\widehat G_{K}\}.
\end{equation}
Moreover, the multiplicity of $\lambda\in\mathcal E(M)$ in $\Spec(\Gamma\ba M,\chi)$ is given by
\begin{equation}\label{eq2:mult_gral}
\mult_{\Delta_{\Gamma,\chi}}(\lambda) := \sum_{\pi\in\widehat G_K(\lambda)} [\chi:\pi|_\Gamma][1_K:\pi|_K].
\end{equation}
In other words, the spectral zeta function of $\Delta_{\Gamma,\chi}$ is given by
\begin{equation}\label{eq2:zeta_gral}
\zeta_{M,\Gamma,\chi}(s) = \sum_{\pi\in \widehat G_K}  [\chi:\pi|_\Gamma]\,[1_K:\pi|_K] \;\lambda(\Cas,\pi)^{-s}.
\end{equation}
\end{theorem}

\begin{proof}
This result is well known, we sketch the proof for completeness.
We consider, on the Hilbert space $L^2(G)$, the representation of $G\times G$ by $L_{g_1}\times R_{g_2}$ where $(L_g\cdot f)(x)=f(g^{-1}x)$ and $(R_g\cdot f)(x)=f(xg)$ for $f\in L^2(G)$.
The Peter-Weyl Theorem implies the decomposition
\begin{equation}\label{eq2:L^2(G)}
L^2(G) \simeq \bigoplus_{\pi\in\widehat G} V_\pi\otimes V_\pi^*
\end{equation}
as $G\times G$-modules.
Here, $v\otimes \varphi\in V_\pi\otimes V_\pi^*$ induces the function $f_{v\otimes \varphi}(x) := \varphi(\pi(x)^{-1}\cdot v)$, therefore the action of $G\times G$ on each component $V_\pi\otimes V_\pi^*$ is given by $({g_1},{g_2})\cdot (v\otimes \varphi)=(\pi(g_1)\cdot v)\otimes (\pi^*(g_2)\cdot \varphi)$.

The subspace of $L^2$-functions on $G$ invariant by $K$ on the right is
\begin{equation}\label{eq2:L^2(G/K)}
L^2(M)=L^2(G/K) \simeq \bigoplus_{\pi\in\widehat G} V_\pi\otimes (V_\pi^*)^K.
\end{equation}
On the other hand, we have the map $\Psi:C^\infty(G/K)\otimes W_\chi\to C^\infty(G/K,W_\chi):=\{\psi:G\to W_\chi\text{ smooth}: \psi(gk)=\psi(g) \;\forall\, g\in G,\,k\in K\}$, given by $\Psi(f,w)(x)=f(x) w$.
The group $\Gamma$ acts on $C^\infty(G/K)$ by the restriction of the action $L_g$ of $G$.
One can check that the map $\Psi$ sends the subspace of $\Gamma$-invariants elements $(C^\infty(G/K)\otimes W_\chi)^\Gamma$ onto $C^\infty(\Gamma\ba G/K;\chi)$.
Hence, $\Gamma^\infty(E_{\Gamma,\chi}) \simeq (C^\infty(G/K)\otimes W_\chi)^\Gamma$.
By taking the closure in the corresponding spaces, we obtain that
\begin{align}\label{eq2:L^2(E_chi)}
L^2(E_{\Gamma,\chi})&\simeq (L^2(G/K)\otimes W_\chi)^\Gamma
\simeq \bigoplus_{\pi\in\widehat G} (V_\pi\otimes W_\chi)^\Gamma\otimes (V_\pi^*)^K \\
&\simeq \bigoplus_{\pi\in\widehat G} \Hom_\Gamma(W_\chi,V_\pi^*)\otimes (V_\pi^*)^K \notag= \bigoplus_{\pi\in\widehat G_K} \Hom_\Gamma(W_\chi,V_\pi)\otimes V_\pi^K. \notag
\end{align}

The Casimir element $\Cas$ also acts on $L^2(G/K)$ leaving invariant each term in \eqref{eq2:L^2(G/K)}.
Moreover, this action commutes with the left action of $G$ on $C^\infty(G/K)$, in particular with elements in $\Gamma$, thus $\Delta_{\Gamma,\chi}$ acts on each term $\Hom_\Gamma(W_\chi,V_\pi)\otimes V_\pi^K$ by the scalar $\lambda(\Cas,\pi)$.
This proves that any eigenvalue of $\Delta_{\Gamma,\chi}$ is $\lambda(\Cas,\pi)$ for some $\pi\in\widehat G_K$.
Moreover, since $\dim\Hom_\Gamma(W_\chi,V_\pi)= [\chi:\pi|_\Gamma]$ and $\dim V_\pi^K=[1_K:\pi|_K]$, \eqref{eq2:L^2(E_chi)} implies the formula for $\mult_{\Delta_{\Gamma,\chi}}(\lambda)$.

The formula for the spectral zeta function follows immediately from the previous description.
This concludes the proof.
\end{proof}

The previous theorem says that in order to describe $\Spec(\Gamma\ba M,\chi)$ one has to calculate three ingredients:
\begin{itemize}
  \item $\widehat G_K$ (with the corresponding numbers $[1_K:\pi|_K]$ for $\pi\in\widehat G_K$);
  \item the numbers $[\chi:\pi|_\Gamma]$ for $\pi\in\widehat G_K$;
  \item the numbers $\lambda(\Cas,\pi)$ for $\pi\in\widehat G_K$.
\end{itemize}
The last one can be easily done by \eqref{eq2:lambda(C,pi)}.
The first one is the most complicated for arbitrary $(G,K)$, though it is well known in our particular cases of interest, namely, compact symmetric spaces of real rank one.
We next deal with the second one, under the assumption of $\Gamma$ is included in a maximal torus $T$ of $G$, taking in mind the situation of $\Gamma$ cyclic.

\subsection{Abelian fundamental group}\label{subsec:abelian}

We associate a maximal torus of $G$ to the Cartan subalgebra $\mathfrak t_\C$ of $\mathfrak g_\C$ as usual.
Let $\mathfrak t_\R=\{X\in\mathfrak t_\C: \alpha(X)\in\R\;\;\forall\,\alpha\in\Sigma(\mathfrak g_\C,\mathfrak t_\C)\}$, and let $T$ be the connected subgroup of $G$ with Lie algebra $\mathfrak t:=\mi \mathfrak t_\R$, thus $T$ is a maximal torus in $G$.
We denote by $P(\mathfrak g_\C)$ the set of weights of $\mathfrak g_\C$ (i.e.\ $\mu\in\mathfrak t_\C^*$ such that $2\langle\mu,\alpha\rangle/\langle\alpha,\alpha\rangle\in\Z$ for all $\alpha\in \Sigma(\mathfrak g_\C,\mathfrak t_\C)$); by $P(G)$ the subset of $\mu\in P(\mathfrak g_\C)$ such that there exists $\xi:T\to \C^\times$ with $\xi(\exp(X))=e^{\mu(X)}$ for all $X\in\mathfrak t$; by $P^{{+}{+}}(\mathfrak g)$ the dominant weights in $P(\mathfrak g)$ and $P^{{+}{+}}(G) = P(G)\cap P^{{+}{+}}(\mathfrak g)$.
For $g\in T$ and $\mu\in P(G)$ let $g^\mu=e^{\mu(X)}$ for any $X\in\mathfrak t$ satisfying $\exp X=g$.

We assume that $\Gamma$ is contained in $T$, thus $\Gamma$ is abelian, so its irreducible representations are one-dimensional (i.e.\ characters).

\begin{lemma}
Let $\Gamma$ be a finite subgroup of $T$.
If $\chi:\Gamma\to\C^\times$ is a morphism and $\pi$ is a finite dimensional representation of $G$, then $[\chi:\pi|_\Gamma]=\dim V_\pi^{\chi(\Gamma)}$, where
\begin{align}\label{eq3:V^chi}
V_\pi^{\chi(\Gamma)}&=\{v\in V_\pi: \pi(\gamma)\cdot v=\chi(\gamma)v\;\forall\,\gamma\in\Gamma\}.
\end{align}
\end{lemma}
\begin{proof}
Since $\Gamma\subset T$, the weight space $V_\pi(\mu)=\{v\in V_\pi: \pi(g)\cdot v= g^\mu v\quad\forall g\in T\}$ is invariant by $\Gamma$ for any $\mu\in P(G)$.
Moreover, as a representation of $\Gamma$, $V_\pi(\mu)$ is equivalent to $m_\pi(\mu)$ copies of the character $\gamma\mapsto \gamma^\mu$, where $m_\pi(\mu):=\dim V_\pi(\mu)$ is the multiplicity of $\mu$ in $\pi$.
Hence, the subrepresentation $V_\pi(\mu)$ of $\Gamma$ contains $\chi$ if and only if $\chi(\gamma)=\gamma^\mu$ for every $\gamma\in\Gamma$, and in this case $[\chi:V_\pi(\mu)|_\Gamma]=m_\pi(\mu)$.
The proof follows since $V_\pi$ is a direct sum of its weight spaces.
\end{proof}

\begin{definition}\label{def:conglattice}
Let $\Gamma$ be a finite subgroup of $T$ and let $\chi:\Gamma\to\C^\times$ be a character.
We associate to $\Gamma$ and $\chi$ the \emph{affine congruence lattice}
\begin{align}\label{eq3:L_Gamma,chi}
\mathcal L_{\Gamma,\chi} &:= \{\mu\in P(G): \gamma^\mu=\chi(\gamma)\quad\forall\,\gamma\in\Gamma\}.
\end{align}
If $\chi=1_\Gamma$, we call it just the \emph{congruence lattice} associated to $\Gamma$ and denote it by $\mathcal L_\Gamma$.
\end{definition}

One can easily see that $\mathcal L_\Gamma$ is a sublattice of $P(G)$ and $\mathcal L_{\Gamma,\chi}$ is a shifted lattice of $\mathcal L_{\Gamma}$, that is, $\mathcal L_{\Gamma,\chi} = \mu+\mathcal L_{\Gamma}$ for any $\mu\in \mathcal L_{\Gamma,\chi}$.

\begin{proposition}\label{prop:dimV^chi}
Let $\Gamma$ be a finite subgroup of $T$ and let $\chi:\Gamma\to\C^\times$ be a character.
For each representation $\pi$ of $G$ we have that
$$
[\chi:\pi|_\Gamma]=\dim V_\pi^{\chi(\Gamma)} = \sum\limits_{\mu\in\mathcal L_{\Gamma,\chi}} m_\pi(\mu).
$$
\end{proposition}

\begin{proof}
We decompose $V_\pi=\oplus_{\mu\in P(G)} V_\pi(\mu)$ into weight spaces.
Let $v=\sum_{\mu\in P(G)}v_\mu \in V_\pi$ with $v_\mu\in V_\pi(\mu)$ for every $\mu\in P(G)$.
Clearly, $v_\mu\neq0$ for only finitely many $\mu$.
We claim that $v\in V_\pi^{\chi(\Gamma)}$ if and only if $v_\mu\in V_\pi^{\chi(\Gamma)}$ for every $\mu\in P(G)$.
Indeed, if $v\in V_\pi^{\chi(\Gamma)}$ then
\begin{align*}
\sum_{\mu\in P(G)} \chi(\gamma)\, v_\mu
=\chi(\gamma) \, v
&=\pi(\gamma)\cdot v
 = \sum_{\mu\in P(G)} \pi(\gamma)\cdot v_\mu
 = \sum_{\mu\in P(G)} \gamma^\mu \, v_\mu
\end{align*}
for every $\gamma\in\Gamma$, thus $\chi(\gamma)=\gamma^\mu$ for every $\mu\in P(G)$ such that $v_\mu\neq0$, hence $v_\mu\in V_\pi(\mu)^{\chi(\Gamma)}$ for every $\mu\in P(G)$.
The converse follows immediately.
We obtain that
\begin{equation*}\label{eq3:decomp-inv}
V_\pi^{\chi(\Gamma)} = \bigoplus_{\mu\in P(G)} V_\pi(\mu)^{\chi(\Gamma)}.
\end{equation*}

By \eqref{eq3:L_Gamma,chi},
$$
V_\pi(\mu)^{\chi(\Gamma)}
=\begin{cases}
V_\pi(\mu) &\quad\text{if $\mu\in\mathcal L_{\Gamma,\chi}$},\\
0 &\quad\text{if $\mu\not\in\mathcal L_{\Gamma,\chi}$}.
\end{cases}
$$
Hence
$$
\dim V_\pi^{\chi(\Gamma)}
= \sum_{\mu\in P(G)} \dim V_\pi(\mu)^{\chi(\Gamma)}
= \sum_{\mu\in \mathcal L_{\Gamma,\chi}} \dim V_\pi(\mu)
= \sum_{\mu\in \mathcal L_{\Gamma,\chi}} m_\pi(\mu)
$$
as asserted.
\end{proof}

\subsection{Cyclic subgroups of classical groups}\label{subsec:classicalgroups}
In this subsection we first fix the root system associated to each classical compact group.
The conventions used are the standard ones, so the reader can skip this part coming back to it for notation when necessary.
We also classify their cyclic subgroups up to conjugation.

\smallskip

\noindent
$\bullet$ \textsl{Type $A_n$}.
Here $G=\SU(n+1)$, $\mathfrak g=\su(n+1)$, $\mathfrak g_\C=\sll(n+1,\C)$,
\begin{align}\label{eq2:maxtorusA}
T&=\{\diag(e^{\mi\theta_1}, \dots,e^{\mi\theta_{n+1}}): \theta_i\in\R\;\forall\,i,\;\textstyle\sum_{i}\theta_i=0\}, \\
\label{eq2:cartansubalgA}
\mathfrak t_\C&=\left\{\diag\big(\theta_1,\dots,\theta_{n+1}\big) : \theta_i\in\C\;\forall\, i,\; \textstyle\sum_i\theta_i=0\right\},
\end{align}
$\varepsilon_i\big(\diag(\theta_1,\dots,\theta_{n+1})\big)=\theta_i$ for each $1\leq i\leq n+1$, $\Sigma^+(\mathfrak g_\C,\mathfrak t_\C)=\{\varepsilon_i-\varepsilon_j: i<j\}$, $\Pi(\mathfrak g_\C,\mathfrak t_\C) = \{\varepsilon_i-\varepsilon_{i+1}: 1\leq i\leq n\}$, $\rho=\sum_{i=1}^{n+1} \frac{n-2i+2}{2}\,\varepsilon_i$, and
\begin{align}\label{eq2:P(G)_G=SU(n+1)}
P(G) &= \left\{\textstyle\sum_{i} a_i\varepsilon_i \in \Z\varepsilon_1\oplus\dots\oplus\Z\varepsilon_{n+1}: \sum_i a_i=0\right\},\\
\label{eq2:P^++(G)_G=SU(n+1)}
P^{{+}{+}}(G) &= \left\{\textstyle\sum_{i}a_i\varepsilon_i \in P(G) :a_1\geq a_2\geq\dots\geq a_{n+1}\right\}.
\end{align}

\smallskip

\noindent
$\bullet$ \textsl{Type $B_n$}.
Here $G=\SO(2n+1)$, $\mathfrak g=\so(2n+1)$, $\mathfrak g_\C=\so(2n+1,\C)$,
\begin{align}\label{eq2:maxtorusB}
T&=
\left\{
\diag\left(
\left[\begin{smallmatrix}\cos(\theta_1)&\sin(\theta_1) \\ -\sin(\theta_1)&\cos(\theta_1)
\end{smallmatrix}\right]
,\dots,
\left[\begin{smallmatrix}\cos(\theta_n)&\sin(\theta_n) \\ -\sin(\theta_n)&\cos(\theta_n)
\end{smallmatrix}\right],1
\right)
:\theta_i\in\R\;\forall\,i
\right\},\\
\label{eq2:cartansubalgB}
\mathfrak t_\C&=
\left\{
\diag\left(
\left[\begin{smallmatrix}0&\theta_1\\ -\theta_1&0\end{smallmatrix}\right]
, \dots,
\left[\begin{smallmatrix}0&\theta_n\\ -\theta_n&0\end{smallmatrix}\right],0
\right):
\theta_i\in\C \;\forall\,i
\right\},
\end{align}
$
\varepsilon_i\big(\diag\left(
\left[\begin{smallmatrix}0&\theta_1\\ -\theta_1&0\end{smallmatrix}\right]
, \dots,
\left[\begin{smallmatrix}0&\theta_n\\ -\theta_n&0\end{smallmatrix}\right],0
\right)\big)=\theta_i
$
for each $1\leq i\leq n$, $\Sigma^+(\mathfrak g_\C,\mathfrak t_\C)=\{\varepsilon_i\pm\varepsilon_j: i<j\}\cup\{\varepsilon_i\}$, $\Pi(\mathfrak g_\C,\mathfrak t_\C) = \{\varepsilon_i-\varepsilon_{i+1}: 1\leq i\leq n-1\}\cup\{\varepsilon_{n}\}$, $\rho=\sum_{i=1}^{n} (n-\frac{2i-1}2)\,\varepsilon_i$,
$
P(G)=\Z\varepsilon_1\oplus\dots\oplus\Z\varepsilon_{n}
$, and
\begin{equation}\label{eq2:P^++(G)_G=SO(2n+1)}
P^{{+}{+}}(G)=\left\{\textstyle\sum_{i}a_i\varepsilon_i \in P(G) :a_1\geq a_{2}\geq \dots\geq a_{n}\geq0\right\}.
\end{equation}

\smallskip

\noindent
$\bullet$ \textsl{Type $C_n$}.
Here $G=\Sp(n,\C)\cap \U(2n)$ where $\Sp(n,\C) = \{g\in\SL(2n,\C): g^t J_ng=J_n:=\left(\begin{smallmatrix}0&\op{Id}_n\\ -\op{Id}_n&0\end{smallmatrix}\right)\}$,
$\mathfrak g=\mathfrak{sp}(n,\C)\cap \mathfrak{u}(2n)$, $\mathfrak g_\C=\mathfrak{sp}(2n,\C)$,
\begin{align}\label{eq2:maxtorusC}
T&=
\left\{
\diag\left(e^{\mi \theta_1},\dots, e^{\mi \theta_n},e^{-\mi \theta_1},\dots, e^{-\mi \theta_n}
\right)
:\theta_i\in\R\;\forall\,i
\right\},\\
\label{eq2:cartansubalgC}
\mathfrak t_\C &=
\left\{
\diag(\theta_1,\dots,\theta_n,-\theta_1,\dots,-\theta_n):
\theta_i\in\C \;\forall\,i
\right\},
\end{align}
$\varepsilon_i\big(\diag\left(\theta_1,\dots,\theta_n,-\theta_1,\dots,-\theta_n \right)\big) =\theta_i$ for each $1\leq i\leq n$, $\Sigma^+(\mathfrak g_\C,\mathfrak t_\C)=\{\varepsilon_i\pm\varepsilon_j: 1\leq i<j\leq n\}\cup\{2\varepsilon_i:1\leq i\leq n\}$, $\Pi(\mathfrak g_\C,\mathfrak t_\C) = \{\varepsilon_i-\varepsilon_{i+1}: 1\leq i\leq n-1\}\cup \{2\varepsilon_n\}$, $\rho=\sum_{i=1}^{n} (n-{i}+1)\,\varepsilon_i$, $P(G)=\Z\varepsilon_1\oplus\dots\oplus\Z\varepsilon_{n}$,
\begin{equation}\label{eq2:P^++(G)_G=Sp(n)}
P^{{+}{+}}(G)=\left\{\textstyle\sum_{i}a_i\varepsilon_i \in P(G) :a_1\geq \dots\geq a_{n-1}\geq a_{n}\geq0\right\}.
\end{equation}

\smallskip

\noindent
$\bullet$ \textsl{Type $D_n$}.
Here $G=\SO(2n)$, $\mathfrak g=\so(2n)$, $\mathfrak g_\C=\so(2n,\C)$,
\begin{align}\label{eq2:maxtorusD}
T&=
\left\{
\diag\left(
\left[\begin{smallmatrix}\cos(\theta_1)&\sin(\theta_1) \\ -\sin(\theta_1)&\cos(\theta_1)
\end{smallmatrix}\right]
,\dots,
\left[\begin{smallmatrix}\cos(\theta_n)&\sin(\theta_n) \\ -\sin(\theta_n)&\cos(\theta_n)
\end{smallmatrix}\right]
\right)
:\theta_i\in\R\;\forall\,i
\right\},\\
\label{eq2:cartansubalgD}
\mathfrak t_\C &=
\left\{
\diag\left(
\left[\begin{smallmatrix}0&\theta_1\\ -\theta_1&0\end{smallmatrix}\right]
, \dots,
\left[\begin{smallmatrix}0&\theta_n\\ -\theta_n&0\end{smallmatrix}\right]
\right):
\theta_i\in\C \;\forall\,i
\right\},
\end{align}
$
\varepsilon_i\big(\diag\left(
\left[\begin{smallmatrix}0&\theta_1\\ -\theta_1&0\end{smallmatrix}\right]
, \dots,
\left[\begin{smallmatrix}0&\theta_n\\ -\theta_n&0\end{smallmatrix}\right]
\right)\big)=\theta_i
$
for each $1\leq i\leq n$, $\Sigma^+(\mathfrak g_\C,\mathfrak t_\C)=\{\varepsilon_i\pm\varepsilon_j: i<j\}$, $\Pi(\mathfrak g_\C,\mathfrak t_\C) = \{\varepsilon_i-\varepsilon_{i+1}: 1\leq i\leq n-1\}\cup\{\varepsilon_{n-1}+\varepsilon_{n}\}$, $\rho=\sum_{i=1}^{n} (n-{i})\,\varepsilon_i$, $P(G)=\Z\varepsilon_1\oplus\dots\oplus\Z\varepsilon_{n}$,
\begin{equation}\label{eq2:P^++(G)_G=SO(2n)}
P^{{+}{+}}(G)=\left\{\textstyle\sum_{i}a_i\varepsilon_i \in P(G) :a_1\geq \dots\geq a_{n-1}\geq|a_{n}|\right\}.
\end{equation}

\medskip

We now classify the cyclic subgroups in the classical compact groups.
Let $G$ be either $\SU(n+1)$, $\SO(2n+1)$, $\Sp(n)$ or $\SO(2n)$, thus the maximal torus $T$ of $G$ fixed above has dimension $n$.
We consider a cyclic subgroup $\Gamma$ of $G$.
Since any element in $G$ is conjugate to an element in $T$, we can assume that $\Gamma$ is included in the maximal torus.
Note that for any closed subgroup $K$ of $G$, the locally homogeneous spaces $\Gamma \ba G/K$ and $(g^{-1}\Gamma g)\ba G/K$ are isometric for any $g\in G$.

For $q\in\N$ and $s=(s_1,\dots,s_n)\in\Z^n$, we set
\begin{equation}\label{eq2:gamma_q,s}
\gamma_{q,s} =
\begin{cases}
\diag(\xi_q^{s_1},\dots, \xi_q^{s_n},\xi_q^{-(s_1+\dots+s_n)})
    &\quad\text{if }G=\SU(n+1),\\[1mm]
\diag(R(\tfrac{2\pi s_1}q),\dots,R(\tfrac{2\pi s_n}q),1)
    &\quad\text{if }G=\SO(2n+1),\\[1mm]
\diag(\xi_q^{s_1},\dots, \xi_q^{s_n},\xi_q^{-s_1},\dots, \xi_q^{-s_n})
    &\quad\text{if }G=\Sp(n),\\[1mm]
\diag(R(\tfrac{2\pi s_1}q),\dots,R(\tfrac{2\pi s_n}q))
    &\quad\text{if }G=\SO(2n),
\end{cases}
\end{equation}
where $\xi_q=e^{2\pi\mi/q}$ and $R(\theta)=\left(\begin{smallmatrix}\cos\theta&\sin\theta\\ -\sin\theta&\cos\theta\end{smallmatrix}\right)$.

\begin{notation}\label{not:s_n+1}
For $s\in\Z^n$, write $s_{n+1}=-(s_1+\dots+s_n)$.
This notation simplifies the case $G=\SU(n+1)$.
The notation $\gcd(q,s)$ means $\gcd(q,s_1,\dots,s_{n+1})$ for $G=\SU(n+1)$ or $\gcd(q,s_1,\dots,s_{n})$ in the rest of the cases.
\end{notation}

Let $q\in\N$ and $s\in\Z^n$ satisfying $\gcd(q,s)=1$.
Clearly, $\gamma_{q,s}$ is in $T$ and has order $q$.
Moreover, any finite subgroup of $T$ of order $q$ is equal to $\Gamma_{q,s}:=\langle\gamma_{q,s}\rangle$ for some $s\in\Z^n$.
The matrix $\gamma_{q,s}$ has eigenvalues $\{\xi^{s_1},\dots,\xi^{s_{n+1}}\}$ for $G=\SU(n+1)$, $\{\xi_q^{\pm s_1},\dots,\xi_q^{\pm s_n},1\}$ for $G=\SO(2n+1)$, and $\{\xi_q^{\pm s_1},\dots,\xi_q^{\pm s_{n}}\}$ for $G=\Sp(n),\SO(2n)$.

When $G=\SO(2n)$ and $\gcd(s_j,q)=1$ for all $j$, the space $\Gamma_{q,s}\ba S^{2n-1}\simeq \Gamma_{q,s}\ba \SO(2n)/\SO(2n-1)$ is a \emph{lens space}.
It is denoted by $L(q;s_1,\dots,s_n)$ in Ikeda's papers and in \cite{LMR15-onenorm}.
Lens spaces are the only manifolds of constant sectional curvature with cyclic fundamental group.
If the condition $\gcd(s_j,q)=1$ does not hold for some $j$, then $\gamma_{q,s}^{q/\gcd(s_j,q)}$ is a non-trivial element in $\Gamma_{q,s}$ having the eigenvalue $1$, thus $\Gamma_{q,s}$ does not act freely on $S^{2n-1}$.
Hence, $\Gamma_{q,s}\ba S^{2n-1}$ is an orbifold with (non-trivial) singularities, so it is not a manifold.
Such orbifolds were called \emph{orbifold lens spaces} in \cite{Shams11}.

\begin{proposition}\label{prop:non-isomP^n(C)}
Let $G=\SU(n+1)$, $q\in\N$ and $s,s'\in\Z^n$ such that $\gcd(q,s)=\gcd(q,s')=1$.
The groups $\Gamma_{q,s}$ and $\Gamma_{q,s'}$ are conjugate in $G$ if and only if there exist $\sigma$ a permutation of $\{1,\dots,n+1\}$ and $\ell\in\Z$ coprime to $q$ satisfying
\begin{equation*}
s_{\sigma(j)}\equiv \ell s_j'\pmod{q} \qquad\text{for all }1\leq j\leq n+1.
\end{equation*}
\end{proposition}
\begin{proof}
The cyclic subgroups $\Gamma_{q,s}$ and $\Gamma_{q,s'}$ are conjugate in $G$ if and only if there exists an integer $\ell$ coprime to $q$ such that $\gamma_{q,s}$ and $\gamma_{q,s'}^\ell$ are conjugate in $G$.
This condition holds if and only if $\gamma_{q,s}$ and $\gamma_{q,s'}^\ell$ have the same multiset of eigenvalues, which is equivalent to $s_{\sigma(j)}\equiv \ell s_j'\pmod q$ for all $j$, for some permutation $\sigma$.
\end{proof}

\begin{proposition}\label{prop:non-isom}
Let $G$ be either $\SO(2n+1)$, $\Sp(n)$ or $\SO(2n)$.
Let $q\in\N$ and $s,s'\in\Z^n$ such that $\gcd(q,s)=\gcd(q,s')=1$.
The groups $\Gamma_{q,s}$ and $\Gamma_{q,s'}$ are conjugate in $G$ if and only if there are $\sigma$ a permutation of $\{1,\dots,n\}$, $\epsilon_1,\dots,\epsilon_n\in\{\pm1\}$, and $\ell\in\Z$ coprime to $q$ satisfying
\begin{equation*}
s_{\sigma(j)}\equiv \epsilon_j\ell s_j'\pmod{q} \qquad\text{for all }1\leq j\leq n.
\end{equation*}
\end{proposition}

\begin{proof}
Similarly as above, $\Gamma_{q,s}$ and $\Gamma_{q,s'}$ are conjugate in $G$ if and only if $\gamma_{q,s}$ and $\gamma_{q,s'}^\ell$ have the same eigenvalues for some integer $\ell$ coprime to $q$, which is equivalent to
$$
\{\xi_q^{\pm s_1},\dots,\xi_q^{\pm s_{n}}\} = \{\xi_q^{\pm \ell s_1'},\dots,\xi_q^{\pm \ell s_{n}'}\}.
$$
The proof thus follows.
\end{proof}

\section{Description of spectra}
We mentioned at the end of \S\ref{subsec:spectra} that the most difficult ingredient to compute $\Spec(\Gamma\ba G/K,\chi)$ is the determination of the $K$-spherical representations $\widehat G_K$.
In \S\ref{subsec:symspaces} we consider compact symmetric spaces of real rank one because $\widehat G_K$ is very simple in these case.
We next describe the spectrum of their corresponding orbifolds in two ways.
First, we give a formula for the multiplicity of the $k$-th eigenvalue in \S\ref{subsec:multiplicities}, and secondly, we give in \S\ref{subsec:generatingfunc} a formula for their corresponding generating function.
We conclude the section by computing the generating functions for cyclic subgroups.

\subsection{Symmetric spaces of real rank one}\label{subsec:symspaces}
For a fuller treatment on symmetric spaces, including their classification, we refer the reader to \cite{Helgason-book-DiffGeomLieGrSymSpaces}.
See also \cite[Ch.\ V]{Helgason-GroupsandGeomAnal} for a description of the spectra of compact symmetric spaces.

Let $M$ be a compact symmetric space of real rank one.
Since we are interested in their quotients, we assume that $M$ is simply connected, so real projective spaces are excluded.
In what follows, we fix the realization of $M=G/K$ as:
\begin{equation}\label{eq3:symRrank1}
\begin{array}{c@{\qquad}c@{\qquad}c@{\qquad}c}
G&K&M&\dim M\\ \hline
\rule{0pt}{14pt}
\SU(n+1) & \op{S}(\U(n)\times\U(1)) & P^n(\C) &2n\\[1mm]
\SO(2n+1) & \SO(2n) &S^{2n}&2n\\[1mm]
\Sp(n) & \Sp(n-1)\times\Sp(1) & P^{n-1}(\Ha)&4n-4\\[1mm]
\SO(2n) & \SO(2n-1) &S^{2n-1}&2n-1\\[1mm]
\op{F}_4 & \Spin(9)& P^{2}(\mathbb O)&16
\end{array}.
\end{equation}
In all classical cases, we consider the Riemannian metric $g_{\textrm{can}}$ on $M$ induced by
\begin{align}\label{eq2:innerproduct}
\langle X,Y\rangle =\tfrac12\op{Tr}(X^*Y).
\end{align}
It is well known that $(S^n,g_{\mathrm{can}})$ has constant sectional curvature one.

Following the notation introduced in \S\ref{subsec:classicalgroups}, let $\Lambda_0\in P(G)$ given by
\begin{equation}
\Lambda_0 =
\begin{cases}
\varepsilon_1-\varepsilon_{n+1}
    &\text{for }G=\SU(n+1),\\
\varepsilon_1
    &\text{for }G=\SO(2n+1), \\
\varepsilon_1+\varepsilon_2
    &\text{for }G=\Sp(n),\\
\varepsilon_1
    &\text{for }G=\SO(2n),\\
\varepsilon_1
    &\text{for }G=\op{F}_4.
\end{cases}
\end{equation}
It is well known that
\begin{equation}\label{eq3:hatG_K_rankone}
\widehat G_{K}=\{\pi_{k\Lambda_0}:k\in\Z_{\geq0}\}.
\end{equation}
Moreover, $[1_K:\pi|_K]=1$ for every $\pi\in\widehat G_{K}$.
In order to write Theorem~\ref{thm:spec_chi_gral} in this particular case, let $\pi_k=\pi_{k\Lambda_0}$ and $\lambda_k=\lambda(\Cas,\pi_k)$, thus $\mathcal E(M)=\{\lambda_k:k\geq0\}$ where
\begin{equation}\label{eq3:lambda_k}
\lambda_k=
\begin{cases}
k(k+n)
    &\quad\text{if $M=P^n(\C)$,}\\
k(k+2n-1)
    &\quad\text{if $M=S^{2n}$,}\\
k(k+2n-1)
    &\quad\text{if $M=P^{n-1}(\Ha)$,}\\
k(k+2n-2)
    &\quad\text{if $M=S^{2n-1}$,}\\
k(k+11)
    &\quad\text{if $M=P^2(\mathbb O)$.}
\end{cases}
\end{equation}

\begin{theorem}\label{thm:spec_chi_rankone}
Let $M=G/K$ be a simply connected compact Riemannian symmetric space of real rank one, let $\Gamma$ be a finite subgroup of $G$ and let $\chi$ be a finite dimensional representation of $\Gamma$.
Then, every eigenvalue in $\Spec(\Gamma\ba M,\chi)$ is in $\{\lambda_k:k\geq0\}$ and $\mult_{\Delta_{M,\Gamma,\chi}}(\lambda_k) := [\chi:\pi_k|_\Gamma]$ for all $k\geq0$.
Equivalently,
\begin{equation}\label{eq2:zeta_rankone}
\zeta_{M,\Gamma,\chi}(s) = \sum_{k\geq0}  \,[\chi:\pi_k|_\Gamma]\;\lambda_k^{-s}.
\end{equation}
\end{theorem}

\subsection{Multiplicity formula}\label{subsec:multiplicities}
By Theorem~\ref{thm:spec_chi_rankone}, the multiplicity $\mult_{\Delta_{M,\Gamma,\chi}}(\lambda_k)=[\chi:\pi_k|_\Gamma]$ is the only remaining ingredient for a given $(\Gamma,\chi)$ to describe $\Spec(\Gamma\ba M,\chi)$.
We will determine it under the assumption that the discrete subgroup $\Gamma$ is contained in the maximal torus of $G$, situation already considered in \S\ref{subsec:abelian}.

Let us denote by $T$ the maximal torus of $G$ fixed in \S\ref{subsec:classicalgroups}.
Assume that $\Gamma$ is a finite subgroup of $T$, thus $\mult_{\Delta_{M,\Gamma,\chi}}(\lambda_k) = [\chi,\pi_k|_\Gamma]$ by Theorem~\ref{thm:spec_chi_rankone}.
Proposition~\ref{prop:dimV^chi} now yields
\begin{equation}\label{eq3:mult_rank1}
\mult_{\Delta_{M,\Gamma,\chi}} (\lambda_k)=\sum_{\mu\in\mathcal L_{\Gamma,\chi}} m_{\pi_k}(\mu).
\end{equation}
The next goal is to calculate the multiplicities $m_{\pi_k}(\mu)$.
There exist several formulas for them, though there do not exist many `closed' formulas as we require.

As it is usual, for $\mu=\sum_i a_i\varepsilon_i\in P(G)$ we denote
\begin{align}
\norma{\mu}_1=\sum_i|a_i| \qquad \text{and}\qquad
\norma{\mu}_\infty=\max_i\, (|a_i|),
\end{align}
the \emph{one-norm} and the \emph{maximum norm} of $\mu$ respectively.
In the literature, $\norma{\cdot}_1$ is also known as \emph{$l^1$-norm} or \emph{taxicab norm}, and $\norma{\cdot}_\infty$ as \emph{$l^\infty$-norm} or \emph{uniform norm}.
We will usually work with $\norma{\cdot}_\infty$ only when $G=\Sp(2)$ and with $\norma{\cdot}_1$ when $G$ is either $\SU(n+1)$, $\SO(2n+1)$ or $\SO(2n)$ for $n\geq2$.
We will sometimes write $\norma{\cdot}$ for $\norma{\cdot}_1$ or $\norma{\cdot}_\infty$ depending on $G$ as above.
We will use the notation $\lfloor x\rfloor$ for the largest integer smaller than or equal to $x$.

\begin{proposition}\label{prop:multiplicities}
Let $k$ be a non-negative integer.
If $G=\SU(n+1)$ and $\mu=\sum_{i=1}^{n+1} a_i\varepsilon_i\in P(G)$, then
\begin{align}\label{eq3:multA}
m_{\pi_{k}}(\mu) =&
\begin{cases}
\displaystyle\binom{k-\frac12 \norma{\mu}_1+n-1}{n-1}
    & \text{ if }\, k-\frac12 \norma{\mu}_1\in \Z_{\geq0},\\[3mm]
0
 & \text{ otherwise.}
\end{cases}
\end{align}
If $G=\SO(2n+1)$ and $\mu=\sum_{i=1}^{n} a_i\varepsilon_i\in P(G)$, then
\begin{align}\label{eq3:multB}
m_{\pi_{k}}(\mu) =&
\begin{cases}
\displaystyle\binom{\lfloor\frac{1}{2}(k-\norma{\mu}_1)\rfloor+n-1}{n-1} & \text{ if }\, k-\norma{\mu}_1\in \Z_{\geq0},\\[3mm]
0&\text{ otherwise.}
\end{cases}
\end{align}
If $G=\Sp(2)$ and $\mu=\sum_{i=1}^{2} a_i\varepsilon_i\in P(G)$, then
\begin{align}\label{eq3:multC}
m_{\pi_{k}}(\mu) =&
\begin{cases}
\displaystyle\lfloor\tfrac{1}{2}(k-\norma{\mu}_\infty)\rfloor +1 & \text{ if }\, k-\norma{\mu}_\infty\in \Z_{\geq0},\text{ and }a_1+a_2\in2\Z, \\[3mm]
0&\text{ otherwise.}
\end{cases}
\end{align}
If $G=\SO(2n)$ and $\mu=\sum_{i=1}^{n} a_i\varepsilon_i\in P(G)$, then
\begin{align}\label{eq3:multD}
m_{\pi_{k}}(\mu) =&
\begin{cases}
\displaystyle\binom{\frac{1}{2}(k-\norma{\mu}_1)+n-2}{n-2} & \text{ if }\, k-\norma{\mu}_1\in 2\Z_{\geq0},\\[3mm]
0 & \text{ otherwise.}
\end{cases}
\end{align}
\end{proposition}

\begin{proof}
Formula \eqref{eq3:multD} has been already proved in \cite[Lem.~3.6]{LMR15-onenorm}, and \eqref{eq3:multB} follows in a very similar way.
Formula \eqref{eq3:multC} is a particular case of the general formula in \cite[Thm.~4.1]{CaglieroTirao04}.

We now suppose that $G=\SU(n+1)$.
Probably this case is already proved in the literature, but we could not find it stated in the form above.
We can assume that $\mu=\sum_{i=1}^{n+1}a_i\varepsilon_i$ is dominant, thus $a_1\geq a_2\geq \dots \geq a_n\geq a_{n+1}$ since the Weyl group in this case is the permutation group in $n+1$ letters, which preserves $\norma{\cdot}_1$.
Suppose that $r:=k-\frac12\norma{\mu}_1\geq0$, thus $\norma{\mu}_1\leq \norma{\Lambda}_1$.
The highest weight $\Lambda_k=k\varepsilon_1-k\varepsilon_{n+1}$ and the dominant weight $\mu$ have associated partitions
\begin{align*}
  \Lambda_k\quad \rightsquigarrow \quad
  k(n+1)&=2k+\underbrace{k+\dots+k}_{n-1}+0, \\
  \mu      \quad \rightsquigarrow \quad
  k(n+1)&=(a_1+k)+\dots+(a_{n+1}+k).
\end{align*}
Hence, $m_{\pi_{k}}(\mu)$ is equal to the number of ways one can fill the Young diagram associated to $\Lambda_k$ with $a_1+k$ $1$'s, $a_2+k$ $2$'s, $\dots$, $a_{n+1}+k$ $(n+1)$'s, in such a way that the entries in each row are non-decreasing and those in each column are strictly increasing (see for instance \cite[page 224]{FultonHarris-book}).
This number is known as the Kostka number of $\mu$ in $\Lambda_k$, that is, the number of semistandard tableaux on $\Lambda_k$ of type $\mu$.
The corresponding Young diagram to $\Lambda_k$ has $2k$ boxes in the first row and $k$ boxes in the next $n-1$ rows.
The remaining combinatorial computation is left to the reader.
\end{proof}

\begin{remark}
The previous lemma tell us that $m_{\pi_k}(\mu)$ depends only on $\norma{\mu}_1$ when $G$ is either $\SU(n+1)$, $\SO(2n+1)$ or $\SO(2n)$.
If $G=\Sp(2)$, by restricting $\mu$ to the lattice
\begin{equation}\label{eq:lattice-even}
D_2:=\{(a_1,a_2)\in\Z^2: a_1+a_2\in2\Z\},
\end{equation}
we obtain that $\mult_{\pi_k}(\mu)$ depends only on $\norma{\mu}_\infty$.

One can check that for $G=\Sp(n)$ with $n\geq3$ and $G=\textrm{F}_4$ the multiplicity formula $\mult_{\pi_k}(\mu)$ cannot be written as above.
An even more complicated situation appears when one considers a compact symmetric space of real rank greater than one since the set of spherical representations cannot be written as in \eqref{eq3:hatG_K_rankone}.
This reason forces us to restrict our attention to orbifolds covered by spheres, complex projective spaces and the quaternion projective line.
\end{remark}

For any $k\geq0$ and for any subset $\mathcal L$ in $P(G)$, let
\begin{equation}\label{eq3:N_L(k)}
N_{\mathcal L}(k) = \#\{\mu\in\mathcal L: \norma{\mu}=k\},
\end{equation}
where $\norma{\cdot}$ stands for $\norma{\cdot}_1$ in all cases except for $G=\Sp(2)$ where $\norma{\cdot}$ denotes $\norma{\cdot}_\infty$.

\begin{theorem}\label{thm:multiplicities}
Let $M=G/K$ be either $P^n(\C)$, $S^{2n}$, $P^{1}(\Ha)$ or $S^{2n-1}$ with $n\geq2$, realized as in \eqref{eq3:symRrank1}.
Let $T$ be the maximal torus in $G$ fixed in \S\ref{subsec:classicalgroups}.
Let $\Gamma$ be a (finite) subgroup of $T$, let $\chi:\Gamma\to\C^\times$ be a character, and let $\mathcal L_{\Gamma,\chi}$ be the associated affine congruence lattice defined in Definition~\ref{def:conglattice}.
Then $\mult_{\Delta_{M,\Gamma,\chi}} (\lambda_k)$ is given by
\begin{equation}\label{eq3:multformula}
\begin{cases}
\displaystyle
\sum_{r=0}^k \tbinom{r+n-1}{n-1} \, N_{\mathcal L_{\Gamma,\chi}}(2k-2r)
    &\quad\text{if $M=P^n(\C)$,}\\[5mm]
\displaystyle
\sum_{r=0}^{\lfloor k/2\rfloor} \tbinom{r+n-1}{n-1} \Big(N_{\mathcal L_{\Gamma,\chi}}(k-2r)+N_{\mathcal L}(k-1-2r)\Big)
    &\quad\text{if $M=S^{2n}$,}\\[5mm]
\displaystyle
\sum_{r=0}^{\lfloor k/2\rfloor} (r+1) \Big(N_{\mathcal L_{\Gamma,\chi}}(k-2r)+ N_{\mathcal L}(k-1-2r)\Big)
    &\quad\text{if $M=P^1(\Ha)$,}\\[5mm]
\displaystyle
\sum_{r=0}^{\lfloor k/2\rfloor} \tbinom{r+n-2}{n-2} N_{\mathcal L_{\Gamma,\chi}}(k-2r)
    &\quad\text{if $M=S^{2n-1}$.}
\end{cases}
\end{equation}
\end{theorem}

\begin{proof}
We first assume $M=P^n(\C)$.
From \eqref{eq3:multA} we have that $m_{\pi_k}(\mu)=0$ if $\norma{\mu}_1>2k$, thus
$$
\mult_{\Delta_{\Gamma,\chi}}(\lambda_k) = \sum_{r=0}^{\lfloor k/2\rfloor}\sum_{\mu\in \mathcal L_{\Gamma,\chi} \atop \norma{\mu}_1=2k-2r} m_{\pi_k}(\mu)
$$
by \eqref{eq3:mult_rank1} and the fact that $\norma{\mu}_1$ is always even.
Note that $m_{\pi_k}(\mu)=\binom{r+n-1}{n-1}$ by \eqref{eq3:multA} for any $\mu$ satisfying $\norma{\mu}_1=2k-2r$.
Since $\binom{r+n-1}{n-1}$ does not depend on $\mu$, the inner sum is equal to $\binom{r+n-1}{n-1}N_{\mathcal L_{\Gamma,\chi}}(2k-2r)$, which proves the formula in this case.

We next assume $M=S^{2n}$.
Again, combining \eqref{eq3:mult_rank1} and \eqref{eq3:multB} we obtain that
\begin{align*}
\mult_{\Delta_{\Gamma,\chi}}(\lambda_k)
&= \sum_{r=0}^{\lfloor k/2\rfloor} \sum_{\mu\in\mathcal L_{\Gamma,\chi}\atop \norma{\mu}=k-2r} \tbinom{r+n-1}{n-1}
 + \sum_{r=0}^{\lfloor (k-1)/2\rfloor} \sum_{\mu\in\mathcal L_{\Gamma,\chi}\atop \norma{\mu}=k-1-2r} \tbinom{r+n-1}{n-1},
\end{align*}
which establishes the formula.
The other cases follow in a very similar way.
\end{proof}

\begin{remark}
The formula for  $M=S^{2n-1}$ and $\chi$ trivial was already obtained in \cite[Thm.~3.5]{LMR15-onenorm}.
\end{remark}

\subsection{Generating functions}\label{subsec:generatingfunc}
The spectral zeta function is not very appropriate for explicit computations in our particular spaces.
See for instance \cite{CarlettiBragadin94a}, \cite{CarlettiBragadin94b}, \cite{Spreafico03}, \cite{Hattori11}, \cite{Teo14}, \cite{BauerFurutani08}, where many technical computations are made to determine its poles, residues and the determinant.
We introduce here a generating function encoding the spectrum of an orbifold covered by a complex projective space, a sphere or the quaternion projective line.
Generating functions were already used by Ikeda in several papers (\cite{Ikeda80_3-dimI}, \cite{Ikeda80_3-dimII}, \cite{Ikeda80_isosp-lens}, \cite{Ikeda83}, \cite{Ikeda88}, \cite{Ikeda97}, \cite{IkedaYamamoto79}), obtaining many examples of isospectral spherical space forms and also spectral rigidity results.

We define the \emph{spectral generating function} associated to $\Delta_{M,\Gamma,\chi}$ or to $\Spec(\Gamma\ba M,\chi)$, by
\begin{equation}\label{eq3:Fdef}
F_{\Delta_{M,\Gamma,\chi}}(z) := \sum_{k\geq0} \mult_{\Delta_{M,\Gamma,\chi}}(\lambda_k)\,z^k.
\end{equation}
We shall abbreviate $F_{{\Gamma,\chi}}(z)$ in place of $F_{\Delta_{M,\Gamma,\chi}}(z)$ when no confusion can arise.
We also write $F_{{\Gamma}}(z)$ when $\chi=1_\Gamma$.

For any subset $\mathcal L$ in $P(G)$, we define the \emph{generating theta function} associated to $\mathcal L$ by
\begin{equation}\label{eq3:vartheta_L}
\vartheta_{\mathcal L}(z) =
\begin{cases}
\displaystyle\sum_{k\geq0} N_{\mathcal L}(2k)\, z^k
    &\quad\text{if }G=\SU(n+1),\\[6mm]
\displaystyle\sum_{k\geq0} N_{\mathcal L}(k)\, z^k
    &\quad\text{if }G=\SO(2n+1),\Sp(2),\SO(2n).
\end{cases}
\end{equation}
The number $N_\mathcal L(k)$  was defined in \eqref{eq3:N_L(k)} and depends on the norm $\norma{\cdot}$ on $P(G)$.
Note that $\norma{\mu}_1$ is even for any $\mu\in P(G)$ when $G=\SU(n+1)$.

\begin{theorem}\label{thm:spectradescription}
Under the same hypotheses of Theorem~\ref{thm:multiplicities}, we have that
\begin{equation}\label{eq3:Fformula}
F_{\Gamma,\chi} (z)=
\begin{cases}
\dfrac{1}{(1-z)^{n}}\; \vartheta_{\mathcal L_{\Gamma,\chi}}(z)
    &\quad\text{if $M=P^n(\C)$,}\\[5mm]
\dfrac{1+z}{(1-z^2)^{n}}\; \vartheta_{\mathcal L_{\Gamma,\chi}}(z)
    &\quad\text{if $M=S^{2n}$,}\\[5mm]
\dfrac{1+z}{(1-z^2)^{2}}\; \vartheta_{D_2\cap\mathcal L_{\Gamma,\chi}}(z)
    &\quad\text{if $M=P^1(\Ha)$,}\\[5mm]
\dfrac{1}{(1-z^2)^{n-1}}\; \vartheta_{\mathcal L_{\Gamma,\chi}}(z)
    &\quad\text{if $M=S^{2n-1}$.}
\end{cases}
\end{equation}
\end{theorem}

\begin{proof}
Assume $M=P^n(\C)$, thus $G=\SU(n+1)$.
Theorem~\ref{thm:multiplicities} implies that
\begin{align*}
F_{\Gamma,\chi}(z)
&= \sum_{k\geq0} \Big(\sum_{r=0}^{k} \tbinom{r+n-1}{n-1} N_{\mathcal L_{\Gamma,\chi}}(2k-2r)\Big)\;z^k\\
&= \Big(\sum_{k\geq0} \tbinom{k+n-1}{n-1}\, z^k\Big)\Big(\sum_{k\geq0} N_{\mathcal L_{\Gamma,\chi}}(2k)z^k\Big)\\
&= \frac{1}{(1-z)^{n}}\;\vartheta_{\mathcal L_{\Gamma,\chi}}(z).
\end{align*}

The rest of the cases are also a direct consequence of Theorem~\ref{thm:multiplicities} and follow in a very similar way.
\end{proof}

\subsection{Cyclic fundamental group case}\label{subsec:cyclic}
We next describe the associated affine congruence lattice $\mathcal L_{\Gamma,\chi}$ for $\Gamma=\Gamma_{q,s}$ as in \S\ref{subsec:classicalgroups}, for $M=P^n(\C),S^{2n},P^1(\Ha),S^{2n-1}$.
We also write the corresponding generating theta function as a rational function in the untwisted case (i.e.\ $\chi$ trivial).

Let $q\in\N$ and $s\in\Z^n$ such that $\gcd(q,s)=1$ (recall the convention in Notation~\ref{not:s_n+1}).
Since $\Gamma_{q,s}$ is cyclic and generated by $\gamma_{q,s}$ as in \eqref{eq2:gamma_q,s}, any character $\chi:\Gamma_{q,s}\to \C^\times$ is determined by the value of $\chi(\gamma_{q,s})$, which is a $q$-root of unity.
When $q$ and $s$ are fixed, for $u\in\Z$, we denote by $\chi_u$ the character of $\Gamma_{q,s}$ such that $\chi_u(\gamma_{q,s})=\xi_q^{u}$.
Clearly, $\{\chi_u:0\leq u<q\}$ is a representative set of $\widehat \Gamma_{q,s}$.

\begin{proposition}\label{prop:L_Gamma,chi_case-by-case}
Let $q\in\N$, $s\in\Z^n$ and $u\in\Z$ such that $\gcd(q,s)=1$.
Then
\begin{equation}
\mathcal L_{\Gamma_{q,s},\chi_{u}} =
\left\{\textstyle\sum_{i}^{} a_i\varepsilon_i\in P(G): \sum_i a_is_i\equiv u\pmod q\right\}.
\end{equation}
Here, the summation runs over $1\leq i\leq n+1$ for $G=\SU(n+1)$ and over $1\leq i\leq n$ for $G=\SO(2n+1),\Sp(n), \SO(2n)$.
\end{proposition}

\begin{proof}
By definition, $\mathcal L_{\Gamma_{q,s},\chi_u}=\{\mu\in P(G): \gamma_{q,s}^\mu = \xi_q^{u}\}$ since $\Gamma_{q,s}$ is generated by $\gamma_{q,s}$.
Let
\begin{equation*}\label{eq3:H_q,s}
H_{q,s} =
\begin{cases}
\diag(2\pi\mi s_1/q,\dots, 2\pi\mi s_{n+1}/q)
    &\quad\text{if }G=\SU(n+1),\\[1mm]
\diag\Big(\left(\begin{smallmatrix}0&2\pi s_1/q\\ -2\pi s_1/q&0\end{smallmatrix}\right),\dots,\left(\begin{smallmatrix}0&2\pi s_{n}/q\\ -2\pi s_{n}/q&0\end{smallmatrix}\right),0\Big)
    &\quad\text{if }G=\SO(2n+1),\\[1mm]
\diag(2\pi\mi s_1/q,\dots, 2\pi\mi s_{n}/q, -2\pi\mi s_1/q,\dots, -2\pi\mi s_{n}/q)
    &\quad\text{if }G=\Sp(n),\\[1mm]
\diag\Big(\left(\begin{smallmatrix}0&2\pi s_1/q\\ -2\pi s_1/q&0\end{smallmatrix}\right),\dots,\left(\begin{smallmatrix}0&2\pi s_{n}/q\\ -2\pi s_{n}/q&0\end{smallmatrix}\right)\Big)
    &\quad\text{if }G=\SO(2n),
\end{cases}
\end{equation*}
thus $\exp(H_{q,s})=\gamma_{q,s}$ in all cases.
Hence, $\gamma_{q,s}^\mu = e^{\mu(H_{q,s})}$ and, since
$$
\mu(H_{q,s}) = \sum_i a_i\varepsilon_i(H_{q,s}) = \sum_i a_i \frac{2\pi \mi s_i}{q} = \frac{2\pi\mi}{q}\sum_i a_is_i
$$
for $\mu=\sum_i a_i\varepsilon_i\in P(G)$, we obtain that $\gamma_{q,s}^\mu=\xi_q^u$ if and only if $\sum_i a_is_i\equiv u\pmod q$.
This completes the proof.
\end{proof}

\begin{notation}\label{not:L_q,s,u}
Let
$
A_n=\{(a_1,\dots,a_{n+1})\in\Z^{n+1}:a_1+\dots+a_{n+1}=0\}.
$
We will make use of the obvious correspondences $P(\SU(n+1))\simeq A_n$ and $P(G)\simeq \Z^{n}$ for $G=\SO(2n),\Sp(n),\SO(2n+1)$ given by $\sum_i a_i\varepsilon_i \leftrightarrow (a_1,a_2,\dots)$.
For $q\in\N$, $s\in\Z^n$ and $u\in\Z$ as above, let
\begin{align}\label{eq3:L_q,s}
\mathcal L_{q,s,u}^* &= \textstyle\left\{(a_1,\dots,a_{n+1})\in A_n: \sum\limits_{i=1}^{n+1} a_is_i\equiv u\pmod q\right\},\\
\mathcal L_{q,s,u}   &= \textstyle\left\{(a_1,\dots,a_{n})\in\Z^{n}: \sum\limits_{i=1}^n a_is_i \equiv u\pmod q\right\}.\notag
\end{align}
Thus, $\mathcal L_{\Gamma_{q,s},\chi_u}$ corresponds to $\mathcal L_{q,s,u}^*$ or to $\mathcal L_{q,s,u}$ depending on $G=\SU(n+1)$ or $G=\SO(2n+1),\Sp(n), \SO(2n)$ respectively.
We write $\mathcal L_{q,s}^*= \mathcal L_{q,s,0}^*$ and $\mathcal L_{q,s}=\mathcal L_{q,s,0}$.

One must be careful with the notation $\mathcal L_{q,s,u}$ since it works for $G=\SO(2n+1), \Sp(2),\SO(2n)$, but $\norma{\cdot}$ on $P(G)$ coincides with $\norma{\cdot}_1$ when $G=\SO(2n+1),\SO(2n)$ and with $\norma{\cdot}_\infty$ when $G=\Sp(2)$.
However, by Theorem~\ref{thm:spectradescription}, one can distinguish their corresponding generating theta function as $\vartheta_{\mathcal L_{q,s,u}}(z)$ for $G=\SO(2n+1),\SO(2n)$ and $\vartheta_{D_2\cap\mathcal L_{q,s,u}}(z)$ for $G=\Sp(2)$.
\end{notation}

Our next goal is, fixing $u=0$ (i.e.\ $\chi_u$ is the trivial character), to give formulas for the generating theta function associated to $\mathcal L_{q,s}^*$, $\mathcal L_{q,s}$ and $D_2\cap\mathcal L_{q,s}$.
Consequently we will obtain that the spectral generating function of an orbifold having cyclic fundamental group covered by complex projective spaces, spheres and the quaternion projective line is a rational function.

We start by writing the theta function as a rational function by using Ehrhart's theory for counting points in rational polytopes.
For more details on Ehrhart's theory see the nice book \cite{BeckRobins-book}.

\begin{theorem}\label{thm:Ehrhart}
Let $q\in\N$ and $s\in\Z^n$ such that $\gcd(q,s)=1$.
Let $\mathcal L$ be either $\mathcal L_{q,s}^*$, $\mathcal L_{q,s}$ for any $n\geq2$ or $D_2\cap\mathcal L_{q,s}$ for $n=2$.
Then, there exist a polynomial $p(z)$ of degree less than $(n+1)q$ such that
\begin{equation}
\vartheta_{\mathcal L}(z) = \frac{(1-z)\,p(z)}{(1-z^q)^{n+1}}.
\end{equation}
\end{theorem}

\begin{proof}
Let
\begin{align*}
\mathcal B &=
\begin{cases}
\left\{x\in\R^{n+1}: \norma{x}_1\leq 2,\; \textstyle\sum_{i=1}^{n+1}x_i=0\right\}
    & \quad\text{if }\mathcal L=\mathcal L_{q,s}^*,\\
\{x\in\R^{n}: \norma{x}_1\leq 1\}
    & \quad\text{if }\mathcal L=\mathcal L_{q,s},\\
\{x\in\R^{2}: \norma{x}_\infty\leq 1\}
    & \quad\text{if }\mathcal L=D_2\cap\mathcal L_{q,s}.
\end{cases}
\end{align*}
Note that $\mathcal B$ is a \emph{cross-polytope} in $\R^n$ in the second case and a cube in $\R^2$ in the third case, both centered at the origin.
Clearly, $N_{\mathcal L}(k)$ is the number of points in $\mathcal L$ at the $k^{\textrm{th}}$ dilate of the  boundary of $\mathcal B$, that is, $N_{\mathcal L}(k) = \#(\mathcal L\cap k\partial\mathcal B)$.
Hence
\begin{equation*}
\vartheta_{\mathcal L}(z) = \sum_{k\geq0}\#(\mathcal L\cap k\partial\mathcal B)\,z^k  = (1-z)\sum_{k\geq0}\#(\mathcal L\cap k\mathcal B)\,z^k.
\end{equation*}

Let $\Psi$ be a linear transformation sending the canonical basis in $\R^n$ to a basis of $\mathcal L$.
Let $\mathcal P=\Psi(\mathcal B)$, hence
$$
\vartheta_{\mathcal L}(z) = (1-z) \sum_{k\geq0}\#(\Z^n\cap k\mathcal P)\,z^k.
$$
The function $\op{Ehr}_{\mathcal P}(z) := \sum_{k\geq0}\#(\Z^n\cap k\mathcal P)z^k$ is called the \emph{Ehrhart series} of $\mathcal P$.
By the Ehrhart Theorem for rational polytopes (see \cite[\S3.7]{BeckRobins-book}), there exists a polynomial $p(z)$ of degree less than $q(n+1)$ such that $\op{Ehr}_{\mathcal P}(z) = {p(z)}/{(1-z^q)^{n+1}}$, which establishes the formula.
\end{proof}

The above theorem implies that the map $\Phi(k):= \sum_{m=0}^k N_{\mathcal L}(m)$, which sends $k$ to the $k^{\textrm{th}}$ coefficient of the series $\vartheta_{\mathcal L}(z)/(1-z)$, is a quasipolynomial on $k$ of degree $n$ and its period divides $q$ (see \cite[Thm.~3.23]{BeckRobins-book}).
The next remark gives the coefficients of $p(z)$ in terms of the coefficients of $\vartheta_{\mathcal L}(z)$, which is useful for explicit computations.

\begin{remark}\label{rem:coef_g(z)}
Write $p(z)=\sum_{j=0}^{q(n+1)-1} h_j z^j$ the polynomial in the proof of Theorem~\ref{thm:Ehrhart}.
By expanding the expression $p(z)= (1-z)^{-1}(1-z^q)^{n+1}\vartheta_{\mathcal L}(z)$ we obtain, for $0\leq k_0<q$ and $0\leq l_0\leq n$, that
\begin{equation*}
h_{k_0+l_0q} = \sum_{j=0}^{l_0} (-1)^j\binom{n+1}{j}\sum_{m=0}^{k_0+q(l_0-j)} N_{\mathcal L}(m).
\end{equation*}
Note that $\sum_{m=0}^{k} N_{\mathcal L}(m)=\#(\Z^n\cap k\mathcal P)$, the $k$-th coefficient of the Ehrhart series $\op{Ehr}_{\mathcal P}(z)$ in the proof of Theorem~\ref{thm:Ehrhart}.
\end{remark}

\begin{remark}\label{rem:q_spectrainvariant}
Theorem~\ref{thm:Ehrhart} shows that the order of the fundamental group $q$ is a spectral invariant.
That is, if two orbifolds with cyclic fundamental group, both covered by $P^n(\C)$, $S^{2n}$, $P^1(\Ha)$ or $S^{2n-1}$, are untwisted isospectral, then their fundamental groups have the same order.
\end{remark}

We next show an alternative way to compute $\vartheta_{\mathcal L_{q,s}}(z)$.

\begin{proposition}\label{prop:Zagier}
Let $q\in\N$ and $s\in\Z^n$ such that $\gcd(q,s)=1$. Then
$$
\vartheta_{\mathcal L_{q,s}}(z) = \frac{(1-z^2)^n}{q}\sum_{l=0}^{q-1} \prod_{j=1}^n \frac{1}{ (z-\xi_q^{ls_j})(z-\xi_q^{-ls_j})}.
$$
\end{proposition}
\begin{proof}
We have
\begin{equation*}
\vartheta_{\mathcal L_{q,s}}(z)
    = \sum_{\mu\in\mathcal L_{q,s}} z^{\norma{\mu}_1}
    =\sum_{(a_1,\dots,a_n)\in\Z^n} \left(\frac1q\sum_{l=0}^{q-1} \xi_q^{l\sum_i a_is_i} \right) z^{\sum_i |a_i|}.
\end{equation*}
By switching the summations, we obtain that
$$
\vartheta_{\mathcal L_{q,s}}(z)= \frac1q \sum_{l=0}^{q-1} \prod_{j=1}^n \left(1+\sum_{m\geq1}\xi_q^{mls_j}z^{ml}+ \sum_{m\geq1}\xi_q^{-mls_j}z^{ml}\right).
$$
Since the quantity inside the parenthesis is equal to $(1-z\xi_q^{ls_j})^{-1}+ (1-z\xi_q^{-ls_j})^{-1}-1= (1-z^2)(1-z\xi_q^{ls_j})^{-1}(1-z\xi_q^{-ls_j})^{-1}$, the assertion follows.
\end{proof}

The previous proof was communicated by Prof.\ Don Zagier to the author at Oberwolfach in July 2013.

\begin{remark}\label{rem:Ikeda-formula}
This proposition and Theorem~\ref{thm:spectradescription} imply that the generating function associated to the lens space $\Gamma_{q,s}\ba S^{2n-1}$ is
\begin{equation}\label{eq3:formulaF_GammaIkeda}
F_{\Gamma_{q,s}}(z) = \frac{1-z^2}{q}\sum_{l=0}^{q-1} \prod_{j=1}^n \frac{1}{ (z-\xi_q^{ls_j})(z-\xi_q^{-ls_j})}.
\end{equation}
This formula has been already proved in \cite[Thm.~3.2]{IkedaYamamoto79}, and was the main tool for Ikeda's construction of isospectral lens spaces (\cite{Ikeda80_isosp-lens}).
\end{remark}

\section{Isospectrality}\label{sec:isospectrality}
This section studies the case when two Laplace operators are \emph{isospectral}, that is, their spectra coincide.
Of course, the notion of isospectrality among orbifolds coincides with the equality of their corresponding spectral generating function.
Now, Theorem~\ref{thm:spectradescription} in our cases of interest, implies that this notion is equivalent to the coincidence of the corresponding generating theta functions.
This characterization is very useful to find examples of twisted and untwisted isospectral orbifolds.

Usually, one says that two orbifolds are isospectral if their corresponding untwisted Laplace operator acting on functions are isospectral.
However, in this paper we will always add the word untwisted to avoid confusion.
Similarly, we will say that two orbifolds are twisted isospectral if there exist nontrivial twists on each of them so that their corresponding twisted Laplace operators are isospectral.

\subsection{Isospectral characterization}\label{subsec:characterization}
Let $G$ be either $\SU(n+1)$, $\SO(2n+1)$, $\Sp(2)$ or $\SO(2n)$.
Recall that $P(G)$ is identified with $A_n$, $\Z^n$, $\Z^2$ and $\Z^n$ respectively (see Notation~\ref{not:L_q,s,u}), and the norm $\norma{\cdot}$ on $P(G)$ is $\norma{\cdot}_1$ excepting for $G=\Sp(2)$ where it is $\norma{\cdot}_\infty$.

\begin{definition}\label{def:norm-isospectral}
Let $\mathcal L$ and $\mathcal L'$ be two subsets of $P(G)$.
We say that they are \emph{$\norma{\cdot}$-isospectral} if $\vartheta_{\mathcal L}(z)=\vartheta_{\mathcal L'}(z)$, or equivalently, $N_{\mathcal L}(k)=N_{\mathcal L'}(k)$ for all $k\geq0$.
More precisely,
$$
\#\{\mu\in\mathcal L: \norma{\mu}=k\} = \#\{\mu\in\mathcal L': \norma{\mu}=k\}
\qquad \forall\, k\geq0.
$$
\end{definition}

The next isospectral characterization is a direct consequence of Theorem~\ref{thm:spectradescription}.

\begin{theorem}\label{thm:characterization}
Let $M=G/K$ be either $P^n(\C)$, $S^{2n}$, $P^1(\Ha)$ or $S^{2n-1}$ for $n\geq2$.
Let $\Gamma$ and $\Gamma'$ be finite subgroups of the maximal torus $T$ of $G$ and let $\chi$ and $\chi'$ be characters of $\Gamma$ and $\Gamma'$ respectively.
If $M\neq P^1(\Ha)$, $\Delta_{\Gamma,\chi}$ and $\Delta_{\Gamma',\chi'}$ are isospectral if and only if $\mathcal L_{\Gamma,\chi}$ and $\mathcal L_{\Gamma',\chi'}$ are $\norma{\cdot}_1$-isospectral.
If $M=P^1(\Ha)$, $\Delta_{\Gamma,\chi}$ and $\Delta_{\Gamma',\chi'}$ are isospectral if and only if $D_2\cap\mathcal L_{\Gamma,\chi}$ and $D_2\cap\mathcal L_{\Gamma',\chi'}$ are $\norma{\cdot}_\infty$-isospectral.
\end{theorem}

\begin{proof}
We have that $\Delta_{\Gamma,\chi}$ and $\Delta_{\Gamma',\chi'}$ are isospectral if and only if $F_{\Gamma,\chi}(z) = F_{\Gamma',\chi'}(z)$, and by Theorem~\ref{thm:spectradescription} this is equivalent to $\vartheta_{\mathcal L_{\Gamma,\chi}}(z) = \vartheta_{\mathcal L_{\Gamma',\chi'}}(z)$ for $M\neq P^1(\Ha)$ or to $\vartheta_{D_2\cap\mathcal L_{\Gamma,\chi}}(z) = \vartheta_{D_2\cap\mathcal L_{\Gamma',\chi'}}(z)$ for $M=P^1(\Ha)$.
The proof is concluded in light of Definition~\ref{def:norm-isospectral}.
\end{proof}

\begin{remark}
An obvious consequence from the previous characterization is that a twisted Laplace operator cannot be isospectral to an untwisted one.
More precisely, $\Delta_{\Gamma,\chi}$ with $\chi$ non-trivial and $\Delta_{\Gamma',1_{\Gamma'}}$ cannot be isospectral since the weight $0$ is in $\mathcal L_{\Gamma',{1_{\Gamma'}}}$ but not in $\mathcal L_{\Gamma,\chi}$.
\end{remark}

The next goal is to show explicit examples of isospectral twisted and untwisted Laplace operators.
We will restrict our attention to the case when $\Gamma$ is a cyclic subgroup of $G$ as in \S\ref{subsec:cyclic}.
By Theorem~\ref{thm:characterization}, following the notation in \eqref{eq3:L_q,s}, it is sufficient, for fixed $q$ and $n$, to check whether $\mathcal L_{q,s,u}$ and $\mathcal L_{q,s',u'}$ (or $\mathcal L_{q,s,u}^*$ and $\mathcal L_{q,s',u'}^*$) are $\norma{\cdot}$-isospectral, that is, $N_{\mathcal L_{q,s,u}}(k) = N_{\mathcal L_{q,s',u'}}(k)$ for any $k\geq0$.

By using the computer program Sage~\cite{Sage}, for low values of $n$ and $q$, we found all pairs $(s,u)$, $(s',u')$ such that $\mathcal L_{q,s,u}$ and $\mathcal L_{q,s',u'}$ (and also $\mathcal L_{q,s,u}^*$ and $\mathcal L_{q,s',u'}^*$) are $\norma{\cdot}$-isospectral.
Clearly, $\mathcal L_{q,s,u}$ and $\mathcal L_{q,s,q-u}$ (and also $\mathcal L_{q,s,u}^*$ and $\mathcal L_{q,s,q-u}^*$) are trivially $\norma{\cdot}$-isospectral, thus we assume $0\leq u\leq q/2$.

\subsection{Isospectral examples covered by spheres}
Table~\ref{table:S^ntwisted} shows families of affine lattices of the form $\mathcal L_{q,s,u}$ that are mutually $\norma{\cdot}_1$-isospectral, for $2\leq n\leq 4$, $q\leq 7$ and $1\leq u\leq \lfloor\frac q2\rfloor$.
They induce families of twisted isospectral non-isometric orbifolds covered by $S^{2n}$ and $S^{2n-1}$.
Table~\ref{table:S^nuntwisted} shows all families of lattices of the form $\mathcal L_{q,s}$ that are mutually $\norma{\cdot}_1$-isospectral, for $2\leq n\leq 5$ and $q\leq 15$.
They induce families of untwisted isospectral non-isometric orbifolds covered by $S^{2n}$ and $S^{2n-1}$.
The cases (assuming $q>2$) that an orbifold covered by a sphere with cyclic fundamental group is a manifold is when $n$ is odd and $\gcd(s_j,q)=1$ for all $1\leq j\leq n$.

\begin{table}
\caption{Twisted $\norma{\cdot}_1$-isospectral families among $\mathcal L_{q,s,u}$ for $2\leq m\leq 4$ and $q\leq 7$}\label{table:S^ntwisted}
\hfill
\begin{minipage}{0.18\textwidth}
$
\begin{array}{ccc}
q&s&u \\ \hline
4 & [0, 1] & 1 \\
4 & [1, 2] & 1 \\ \hline
5 & [0, 1] & 1 \\
5 & [1, 2] & 1 \\
5 & [1, 2] & 2 \\ \hline
7 & [1, 2] & 1 \\
7 & [1, 2] & 2 \\ \hline \hline
4 & [0, 0, 1] & 1 \\
4 & [0, 1, 2] & 1 \\
4 & [1, 2, 2] & 1 \\ \hline
\end{array}
$
\end{minipage}
\hfill
\begin{minipage}[t]{0.18\textwidth}
$
\begin{array}{ccc}
q&s&u \\ \hline
4 & [0, 1, 1] & 1 \\
4 & [1, 1, 2] & 1 \\ \hline
5 & [0, 0, 1] & 1 \\
5 & [0, 1, 2] & 1 \\
5 & [0, 1, 2] & 2 \\ \hline
7 & [0, 1, 2] & 1 \\
7 & [0, 1, 2] & 2 \\
7 & [1, 2, 3] & 1 \\
7 & [1, 2, 3] & 2 \\
7 & [1, 2, 3] & 3 \\ \hline\hline
\end{array}
$
\end{minipage}
\hfill
\begin{minipage}[t]{0.20\textwidth}
$
\begin{array}{ccc}
q&s&u \\ \hline
4 & [0, 0, 0, 1] & 1 \\
4 & [0, 0, 1, 2] & 1 \\
4 & [0, 1, 2, 2] & 1 \\
4 & [1, 2, 2, 2] & 1 \\ \hline
4 & [0, 0, 1, 1] & 1 \\
4 & [0, 1, 1, 2] & 1 \\
4 & [1, 1, 2, 2] & 1 \\ \hline
4 & [0, 1, 1, 1] & 1 \\
4 & [1, 1, 1, 2] & 1 \\ \hline
\\
\end{array}
$
\end{minipage}
\hfill
\begin{minipage}[t]{0.20\textwidth}
$
\begin{array}{ccc}
q&s&u \\ \hline
5 & [0, 0, 0, 1] & 1 \\
5 & [0, 0, 1, 2] & 1 \\
5 & [0, 0, 1, 2] & 2 \\ \hline
5 & [0, 1, 1, 2] & 1 \\
5 & [1, 1, 2, 2] & 1 \\
5 & [1, 1, 2, 2] & 2 \\ \hline
7 & [1, 1, 2, 2] & 1 \\
7 & [1, 1, 2, 3] & 1 \\ \hline
\\ \\
\end{array}
$
\end{minipage}
\hfill
\begin{minipage}[t]{0.20\textwidth}
$
\begin{array}{ccc}
q&s&u \\ \hline
7 & [0, 0, 1, 2] & 1 \\
7 & [0, 0, 1, 2] & 2 \\
7 & [0, 1, 2, 3] & 1 \\
7 & [0, 1, 2, 3] & 2 \\
7 & [0, 1, 2, 3] & 3 \\ \hline \hline
\\ \\ \\ \\ \\
\end{array}
$
\end{minipage}
\hfill
\end{table}

\begin{table}
\caption{Untwisted $\norma{\cdot}_1$-isospectral families among $\mathcal L_{q,s}$ for $2\leq m\leq 5$ and $q\leq 15$}\label{table:S^nuntwisted}
\rule{0pt}{0pt}\hfill
\begin{minipage}[t]{0.18\textwidth}
$
\begin{array}{cc}
q&s \\ \hline
11 & [1, 2, 3] \\
11 & [1, 2, 4] \\ \hline
13 & [1, 2, 3] \\
13 & [1, 2, 4] \\ \hline
13 & [1, 2, 5] \\
13 & [1, 3, 4] \\ \hline
15 & [1, 2, 6] \\
15 & [1, 3, 4] \\ \hline \hline
\end{array}
$
\end{minipage}\hfill
\begin{minipage}[t]{0.18\textwidth}
$
\begin{array}{cc}
q&s \\ \hline
11 & [0, 1, 2, 3] \\
11 & [0, 1, 2, 4] \\ \hline
13 & [0, 1, 2, 3] \\
13 & [0, 1, 2, 4] \\ \hline
13 & [0, 1, 2, 5] \\
13 & [0, 1, 3, 4] \\ \hline
13 & [1, 2, 3, 4] \\
13 & [1, 2, 3, 5] \\ \hline
\end{array}
$
\end{minipage} \hfill
\begin{minipage}[t]{0.2\textwidth}
$
\begin{array}{cc}
q&s \\ \hline
15 & [0, 1, 2, 6] \\
15 & [0, 1, 3, 4] \\ \hline
15 & [1, 2, 5, 6] \\
15 & [1, 3, 4, 5] \\ \hline\hline
11 & [0, 0, 1, 2, 3] \\
11 & [0, 0, 1, 2, 4] \\ \hline
13 & [0, 0, 1, 2, 3] \\
13 & [0, 0, 1, 2, 4] \\ \hline
\end{array}
$
\end{minipage}\hfill
\begin{minipage}[t]{0.2\textwidth}
$
\begin{array}{cc}
q&s \\ \hline
13 & [0, 0, 1, 2, 5] \\
13 & [0, 0, 1, 3, 4] \\ \hline
13 & [0, 1, 2, 3, 4] \\
13 & [0, 1, 2, 3, 5] \\ \hline
15 & [0, 0, 1, 2, 6] \\
15 & [0, 0, 1, 3, 4] \\ \hline
\\ \\
\end{array}
$
\end{minipage}\hfill
\begin{minipage}[t]{0.20\textwidth}
$
\begin{array}{cc}
q&s \\ \hline
15 & [0, 1, 2, 5, 6] \\
15 & [0, 1, 3, 4, 5] \\ \hline
15 & [1, 2, 5, 5, 6] \\
15 & [1, 3, 4, 5, 5] \\ \hline
15 & [1, 2, 3, 6, 6] \\
15 & [1, 3, 3, 4, 6] \\ \hline\hline
\\ \\
\end{array}
$
\end{minipage}\hfill
\end{table}

Here are some remarks and particular examples.

\begin{remark}
Ikeda~\cite{Ikeda80_isosp-lens} showed a simple characterization and several examples of untwisted isospectral lens spaces by using \eqref{eq3:formulaF_GammaIkeda}.
His method works for $q$ prime, but it can be extended in general.
The article \cite{GornetMcGowan06} contains a numerical study of isospectral lens spaces based on Ikeda's method, though their results are correct for $q$ prime but not in general.
\end{remark}

\begin{remark}\label{rem:Shams}
Shams~\cite{Shams11} extended Ikeda's method to orbifold lens spaces, that is, spaces of the form $\Gamma_{q,s}\ba S^{2n-1}$ with $\gcd(q,s_j)\neq1$ for some $1\leq j\leq n$.
He obtained pairs of isospectral non-isometric orbifold lens spaces in any dimension greater than or equal to $9$.
The pair $\mathcal L_{15;(1,2,6)}$, $\mathcal L_{15;(1,3,4)}$ in Table~\ref{table:S^nuntwisted} gives an example in dimension $5$ and $6$, and the same table contains several examples in dimensions $7$ and $8$.
The author believes that such examples do not exist in any dimension smaller than $5$, even in the case of non-cyclic fundamental group.
\end{remark}

\begin{conjecture}\label{conj:dim3-4}
There are no pairs of untwisted isospectral non-isometric orbifolds covered by $S^n$ with $n\leq 4$.
\end{conjecture}

\begin{remark}\label{rem:supportconjecture}
In support to the previous conjecture, Ikeda and Yamamoto \cite{IkedaYamamoto79} \cite{Yamamoto80} have already proved that there are no untwisted isospectral non-isometric $3$-dimensional lens spaces.
In other words, they proved the conjecture for manifolds covered by $S^3$.
Our computations did not find any untwisted isospectral non-isometric $3$ or $4$-dimensional orbifold lens spaces with $q\leq 200$.
However, Table~\ref{table:S^ntwisted} shows twisted isospectral $3$ and $4$-dimensional orbifold lens spaces.
\end{remark}

\begin{example}\label{ex:L_4(0,1)1L_4(1,2)1}
Here we give a simple proof of $\norma{\cdot}_1$-isospectrality for the first example in Table~\ref{table:S^ntwisted}, namely $\mathcal L:=\mathcal L_{4,(0,1),1}$ and $\mathcal L':=\mathcal L_{4,(1,2),1}$.
We have that
$$
\mathcal L=\{(a,b)\in\Z^2: b\equiv 1\pmod 4\},
\qquad
\mathcal L'=\{(a,b)\in\Z^2: a+2b\equiv 1\pmod 4\}.
$$
Figure~\ref{fig:L_4(0,1)(1,2)} shows a picture of them.
Let $\varphi:\mathcal L'\to\Z^2$ given by
$$
\varphi(a,b)=
\begin{cases}
(a,b) &\quad\text{if }b\equiv 1\pmod4,\\
(-a,-b) &\quad\text{if }b\equiv 3\pmod4,\\
(b,a) &\quad\text{if }b\equiv 0\pmod2.
\end{cases}
$$
Clearly, $\varphi$ is $\norma{\cdot}_1$-preserving and one can easily check that it is injective and $\varphi(\mathcal L') = \mathcal L$.
It follows immediately that $\mathcal L$ and $\mathcal L'$ are $\norma{\cdot}_1$-isospectral.
\end{example}

\begin{remark}\label{rem:twistediso-mfd-orb}
Dryden, Gordon, Greenwald and Webb~\cite{DrydenGordonGreenwaldWebb08} proved that a manifold cannot be untwisted isospectral to an orbifold (that is not a manifold) when they have a common Riemannian cover.
The second family in Table~\ref{table:S^ntwisted} shows that this is not true for twisted isospectrality.
Indeed, the pair $\mathcal L_{5,(0,1),1},\,\mathcal L_{5,(1,2),1}$ induces the twisted isospectral pair $(\Gamma_{5,(0,1)}\ba S^{3},\chi_1)$, $(\Gamma_{5,(1,2)}\ba S^{3},\chi_1)$, where $\Gamma_{5,(1,2)}\ba S^{3}$ is a manifold and $\Gamma_{5,(0,1)}\ba S^{3}$ is an orbifold.
\end{remark}

\begin{example}
Here we give a proof of $\norma{\cdot}_1$-isospectrality of $\mathcal L_{q,(1,2),1}$ and $\mathcal L_{q,(1,2),2}$ for every $q$ odd.
Such as in Example~\ref{ex:L_4(0,1)1L_4(1,2)1}, we will give a $\norma{\cdot}_1$-preserving bijection between them.

Let $\mathcal L=\mathcal L_{q,(1,2)}=\{(a,b)\in\Z^2:a+2b\equiv0\pmod q\}$ (see Figure~\ref{fig:L_7(1,2)} for $q=7$), thus $\mathcal L_{q,(1,2),1}= \mathcal L + (1,0)$ and $\mathcal L_{q,(1,2),2}= \mathcal L + (0,1)$.
A nice way to see the points in these affine lattices is to see the points in $\mathcal L$ moving the origin to $(-1,0)$ and $(0,-1)$ respectively.
Hence, we are looking for a bijection $\varphi:\mathcal L\to\mathcal L$ such that $\norma{\varphi(\mu)-(0,-1)}_1=\norma{\mu-(-1,0)}_1$ for every $\mu\in\mathcal L$.
One can check that this condition is satisfied by the function
$$
\varphi(a,b)=
\begin{cases}
(a,b)   \quad&\text{if } a,b\geq1,\text{ or }a,b\leq -1,\\
(-a,-b) \quad&\text{if }-a,b\geq1,\text{ or }-a,b\leq -1,\\
(0,a)   \quad&\text{if } a\leq -1\text{ and }b=0.
\end{cases}
$$

One can easily check that $\mathcal L_{q,(0^m,1,2,),1}$ and $\mathcal L_{q,(0^m,1,2),2}$ are also $\norma{\cdot}_1$-isospectral for every $q$ odd.
Here $0^m$ in $s$ and $s'$ means a zero repeated $m$ times.
\end{example}

\begin{remark}
One can check that, for a fixed isospectral pair $(q,s,u)$ and $(q,s',u')$, if we add zero coordinates to $s$ and $s'$, then we again obtain isospectrality.
This fact was already known (see \cite[Cor.~3.2.4]{Shams11}).
\end{remark}

\setlength{\unitlength}{5mm}

\begin{figure}
\hfill
\begin{minipage}{0.40\textwidth}
\begin{picture}(13,13)
\multiput(0,0)(0,1){13}{\multiput(0,0)(1,0){13}{\circle*{0.1}}}
\put(6,0){\line(0,1){12}}
\put(0,6){\line(1,0){12}}
\multiput(0,0)(0,2){7}{\multiput(3,0)(4,0){3}{\circle{0.2}}}
\multiput(0,0)(0,2){6}{\multiput(1,1)(4,0){3}{\circle{0.2}}}
\multiput(0,0)(0,4){3}{\multiput(-0.2,3)(1,0){13}{\line(1,0){0.4}}}
\multiput(0,0)(0,4){3}{\multiput(0,2.8)(1,0){13}{\line(0,1){0.4}}}
\end{picture}
\caption{$\mathcal L_{4,(0,1),1},\,\mathcal L_{4,(1,2),1}$  }
\label{fig:L_4(0,1)(1,2)}
\end{minipage}\hfill
\begin{minipage}{0.40\textwidth}
\begin{picture}(13,13)
\multiput(0,0)(0,1){13}{\multiput(0,0)(1,0){13}{\circle*{0.1}}}
\put(6,0){\line(0,1){12}}
\put(0,6){\line(1,0){12}}
\multiput(0,0)(0,7){2}{\multiput(4,0)(7,0){2}{\circle*{0.3}}}
\multiput(0,0)(0,7){2}{\multiput(2,1)(7,0){2}{\circle*{0.3}}}
\multiput(0,0)(0,7){2}{\multiput(0,2)(7,0){2}{\circle*{0.3}}}
\multiput(0,0)(0,7){2}{\multiput(5,3)(7,0){2}{\circle*{0.3}}}
\multiput(0,0)(0,7){2}{\multiput(3,4)(7,0){2}{\circle*{0.3}}}
\multiput(0,0)(0,7){2}{\multiput(1,5)(7,0){2}{\circle*{0.3}}}
\multiput(6,6)(7,0){1}{\circle*{0.3}}
\put(5,6){\circle{0.4}}
\put(6,5){\circle{0.4}}
\end{picture}
\caption{$\mathcal L_{7,(1,2)}$}\label{fig:L_7(1,2)}
\end{minipage}
\hfill
\end{figure}

\subsection{Isospectral examples covered by complex projective spaces}
Table~\ref{table:P^n(C)twisted} shows all families of affine lattices of the form $\mathcal L_{q,s,u}^*$ that are mutually $\norma{\cdot}_1$-isospectral, for $2\leq n\leq 3$, $q\leq 6$ and $1\leq u\leq \lfloor\frac q2\rfloor$.
They induce families of $2n$-dimensional twisted isospectral non-isometric orbifolds covered by $P^n(\C)$.
Table~\ref{table:P^n(C)untwisted} shows all families of lattices of the form $\mathcal L_{q,s}^*$ that are mutually $\norma{\cdot}_1$-isospectral, for $2\leq n\leq 3$ and $q\leq 10$.
They induce families of $2n$-dimensional untwisted isospectral non-isometric orbifolds covered by $P^n(\C)$.
Every orbifold in this case is not a manifold.

\begin{table}
\caption{Twisted $\norma{\cdot}_1$-isospectral families among $\mathcal L_{q,s,u}^*$ for $2\leq n\leq 3$ and $q\leq 6$}\label{table:P^n(C)twisted}
\hfill
\begin{minipage}{0.18\textwidth}
$
\begin{array}{ccc}
q&s&u \\ \hline
4 & [0, 1, 3] & 1 \\
4 & [1, 1, 2] & 1 \\ \hline
6 & [0, 1, 5] & 1 \\
6 & [1, 2, 3] & 1 \\ \hline
6 & [0, 1, 5] & 2 \\
6 & [1, 2, 3] & 2 \\ \hline
6 & [0, 1, 5] & 3 \\
6 & [1, 2, 3] & 3 \\ \hline
\\
\end{array}
$
\end{minipage}
\hfill
\begin{minipage}[t]{0.20\textwidth}
$
\begin{array}{ccc}
q&s&u \\ \hline
6 & [1, 1, 4] & 1 \\
6 & [1, 1, 4] & 2 \\ \hline \hline
4 & [0, 0, 1, 3] & 1 \\
4 & [0, 1, 1, 2] & 1 \\
4 & [1, 2, 2, 3] & 1 \\ \hline
4 & [0, 0, 1, 3] & 2 \\
4 & [0, 1, 1, 2] & 2 \\
4 & [1, 2, 2, 3] & 2 \\ \hline
\\
\end{array}
$
\end{minipage}
\hfill
\begin{minipage}[t]{0.20\textwidth}
$
\begin{array}{ccc}
q&s&u \\ \hline
4 & [1, 1, 1, 1] & 1 \\
4 & [1, 1, 1, 1] & 2 \\
4 & [1, 1, 3, 3] & 1 \\ \hline
5 & [1, 2, 3, 4] & 1 \\
5 & [1, 2, 3, 4] & 2 \\ \hline
6 & [0, 0, 1, 5] & 1 \\
6 & [2, 3, 3, 4] & 1 \\ \hline
6 & [0, 0, 1, 5] & 2 \\
6 & [2, 3, 3, 4] & 2 \\ \hline
\end{array}
$
\end{minipage}
\hfill
\begin{minipage}[t]{0.20\textwidth}
$
\begin{array}{ccc}
q&s&u \\ \hline
6 & [0, 0, 1, 5] & 3 \\
6 & [2, 3, 3, 4] & 3 \\ \hline
6 & [1, 1, 1, 3] & 1 \\
6 & [1, 1, 1, 3] & 3 \\
6 & [1, 1, 5, 5] & 1 \\
6 & [1, 1, 5, 5] & 3 \\
6 & [1, 3, 3, 5] & 1 \\
6 & [1, 3, 3, 5] & 3 \\ \hline
\\
\end{array}
$
\end{minipage}
\hfill
\end{table}

\begin{table}
\caption{Untwisted $\norma{\cdot}_1$-isospectral families among $\mathcal L_{q,s}^*$ for $2\leq n\leq 3$ and $q\leq 10$}\label{table:P^n(C)untwisted}
\rule{0pt}{0pt}\hfill
\begin{minipage}[t]{0.18\textwidth}
$
\begin{array}{cc}
q&s \\ \hline
6 & [0, 1, 5] \\
6 & [1, 2, 3] \\ \hline
9 & [0, 1, 8] \\
9 & [1, 2, 6] \\ \hline\hline
\\
\end{array}
$
\end{minipage}\hfill
\begin{minipage}[t]{0.18\textwidth}
$
\begin{array}{cc}
q&s \\ \hline
4 & [0, 0, 1, 3] \\
4 & [0, 1, 1, 2] \\
4 & [1, 2, 2, 3] \\ \hline
6 & [0, 0, 1, 5] \\
6 & [2, 3, 3, 4] \\ \hline
\end{array}
$
\end{minipage} \hfill
\begin{minipage}[t]{0.2\textwidth}
$
\begin{array}{cc}
q&s \\ \hline
7 & [0, 1, 2, 4] \\
7 & [1, 2, 5, 6] \\ \hline
8 & [0, 0, 1, 7] \\
8 & [1, 2, 2, 3] \\
8 & [1, 4, 4, 7] \\ \hline
\end{array}
$
\end{minipage}\hfill
\begin{minipage}[t]{0.2\textwidth}
$
\begin{array}{cc}
q&s \\ \hline
8 & [0, 1, 1, 6] \\
8 & [1, 1, 2, 4] \\ \hline
8 & [1, 1, 3, 3] \\
8 & [1, 1, 7, 7] \\ \hline
\\
\end{array}
$
\end{minipage}\hfill
\begin{minipage}[t]{0.20\textwidth}
$
\begin{array}{cc}
q&s \\ \hline
10 & [0, 0, 1, 9] \\
10 & [2, 5, 5, 8] \\ \hline
10 & [0, 1, 1, 8] \\
10 & [1, 2, 2, 5] \\ \hline
\\
\end{array}
$
\end{minipage}\hfill
\end{table}

\begin{remark}
We explain here the first example in Table~\ref{table:P^n(C)untwisted} and the second, third and fourth example in Table~\ref{table:P^n(C)twisted}.
Let $q=6$, $s=(0,1,5)$ and $s'=(1,2,3)$.
One can check that
\begin{align*}
\mathcal L &:=\mathcal L_{q,s}^* \,= \Z(6,0,-6)\oplus\Z(-2,1,1)=\{(a,b,c)\in A_2: b\equiv c\pmod 6\}  ,\\
\mathcal L'&:=\mathcal L_{q,s'}^* =\Z(0,6,-6)\oplus\Z(1,-2,1)=\{(a,b,c)\in A_2: -b\equiv 2c\pmod 6\}.
\end{align*}
It follows immediately that $\mathcal L_{q,s,u}^*$ and $\mathcal L_{q,s',u}^*$ are $\norma{\cdot}_1$-isospectral for any $u$ since $\mathcal L$ and $\mathcal L'$ are related by the permutation $(a,b,c)\mapsto (b,a,c)$ which clearly preserves $\norma{\cdot}_1$.
On the other hand, the orbifolds $\Gamma_{q,s}\ba P^2(\C)$ and $\Gamma_{q,s'}\ba P^2(\C)$ have dimension $4$ and are non-isometric by Proposition~\ref{prop:non-isomP^n(C)} since every coefficient in $s'$ is not divisible by $q$.
\end{remark}

\begin{remark}
The minimum dimension of twisted or untwisted orbifold covered by complex projective spaces is $4$, thus $n=2$.
Indeed, if $n=1$ then $\mathcal L_{q,s}=\Z(q,-q)$ for any $s\in\Z$ coprime to $q$, thus there is only one such orbifold whose fundamental group order is $q$, and the assertion follows since $q$ is a spectral invariant (see Remark~\ref{rem:q_spectrainvariant}).
\end{remark}

\subsection{Isospectral examples covered by the quaternion projective line}
Table~\ref{table:P^1(H)twisted} shows all families of affine lattices of the form $\mathcal L_{q,s,u}$ that are mutually $\norma{\cdot}_\infty$-isospectral, for $n=2$, $q\leq 9$ and $1\leq u\leq \lfloor\frac q2\rfloor$.
They induce families of $4$-dimensional twisted isospectral orbifolds covered by $P^1(\Ha)$.
Table~\ref{table:P^1(H)untwisted} shows all families of lattices of the form $\mathcal L_{q,s}$ that are mutually $\norma{\cdot}_\infty$-isospectral, for $n=2$ and $q\leq 10$.
They induce families of $4$-dimensional untwisted isospectral non-isometric orbifolds covered by $P^1(\Ha)$.
Similarly as in the case of $P^n(\C)$, every orbifold in this subsection is not a manifold.

\begin{table}
\caption{Twisted $\norma{\cdot}_\infty$-isospectral families among $\mathcal L_{q,s,u}$ for $n=2$ and $q\leq 9$}\label{table:P^1(H)twisted}
\hfill
\begin{minipage}{0.18\textwidth}
$
\begin{array}{ccc}
q&s&u \\ \hline
4 & [0, 1] & 1 \\
4 & [1, 2] & 1 \\ \hline
4 & [0, 1] & 2 \\
4 & [1, 2] & 2 \\ \hline
5 & [1, 1] & 2 \\
5 & [1, 2] & 1 \\
5 & [1, 2] & 2 \\ \hline
6 & [0, 1] & 1 \\
6 & [2, 3] & 1 \\ \hline
\end{array}
$
\end{minipage}
\hfill
\begin{minipage}[t]{0.18\textwidth}
$
\begin{array}{ccc}
q&s&u \\ \hline
6 & [0, 1] & 2 \\
6 & [2, 3] & 2 \\ \hline
6 & [0, 1] & 3 \\
6 & [2, 3] & 3 \\ \hline
6 & [1, 1] & 1 \\
6 & [1, 1] & 3 \\
6 & [1, 3] & 1 \\
6 & [1, 3] & 3 \\ \hline
\\
\end{array}
$
\end{minipage}
\hfill
\begin{minipage}[t]{0.18\textwidth}
$
\begin{array}{ccc}
q&s&u \\ \hline
7 & [1, 2] & 1 \\
7 & [1, 2] & 3 \\ \hline
8 & [0, 1] & 1 \\
8 & [1, 4] & 3 \\ \hline
8 & [0, 1] & 2 \\
8 & [1, 2] & 2 \\
8 & [1, 4] & 2 \\ \hline
8 & [0, 1] & 3 \\
8 & [1, 4] & 1 \\ \hline
\end{array}
$
\end{minipage}
\hfill
\begin{minipage}[t]{0.18\textwidth}
$
\begin{array}{ccc}
q&s&u \\ \hline
8 & [0, 1] & 4 \\
8 & [1, 4] & 4 \\ \hline
8 & [1, 1] & 1 \\
8 & [1, 1] & 3 \\
8 & [1, 3] & 1 \\
8 & [1, 3] & 3 \\ \hline
8 & [1, 1] & 2 \\
8 & [1, 3] & 2 \\ \hline
\\
\end{array}
$
\end{minipage}
\hfill
\begin{minipage}[t]{0.18\textwidth}
$
\begin{array}{ccc}
q&s&u \\ \hline
8 & [1, 2] & 1 \\
8 & [1, 2] & 3 \\ \hline
9 & [1, 1] & 4 \\
9 & [1, 2] & 4 \\ \hline
9 & [1, 3] & 2 \\
9 & [1, 3] & 4 \\ \hline
\\ \\ \\
\end{array}
$
\end{minipage}
\hfill
\end{table}

\begin{table}
\caption{Untwisted $\norma{\cdot}_\infty$-isospectral families among $\mathcal L_{q,s}$ for $n=2$ and $q\leq 20$}\label{table:P^1(H)untwisted}
\rule{0pt}{0pt}\hfill
\begin{minipage}[t]{0.125\textwidth}
$
\begin{array}{cc}
q&s \\ \hline
4 & [0, 1] \\
4 & [1, 2] \\ \hline
6 & [0, 1] \\
6 & [2, 3] \\ \hline
\end{array}
$
\end{minipage}\hfill
\begin{minipage}[t]{0.125\textwidth}
$
\begin{array}{cc}
q&s \\ \hline
8 & [0, 1] \\
8 & [1, 4] \\ \hline
10 & [0, 1] \\
10 & [2, 5] \\ \hline
\end{array}
$
\end{minipage} \hfill
\begin{minipage}[t]{0.125\textwidth}
$
\begin{array}{cc}
q&s \\ \hline
12 & [0, 1] \\
12 & [1, 6] \\ \hline
12 & [1, 2] \\
12 & [1, 4] \\ \hline
\end{array}
$
\end{minipage} \hfill
\begin{minipage}[t]{0.125\textwidth}
$
\begin{array}{cc}
q&s \\ \hline
12 & [2, 3] \\
12 & [3, 4] \\ \hline
14 & [0, 1] \\
14 & [2, 7] \\ \hline
\end{array}
$
\end{minipage}\hfill
\begin{minipage}[t]{0.125\textwidth}
$
\begin{array}{cc}
q&s \\ \hline
14 & [1, 2] \\
14 & [1, 4] \\ \hline
16 & [0, 1] \\
16 & [1, 8] \\ \hline
\end{array}
$
\end{minipage}\hfill
\begin{minipage}[t]{0.125\textwidth}
$
\begin{array}{cc}
q&s \\ \hline
16 & [1, 2] \\
16 & [1, 6] \\ \hline
18 & [0, 1] \\
18 & [2, 9] \\ \hline
\end{array}
$
\end{minipage}\hfill
\begin{minipage}[t]{0.125\textwidth}
$
\begin{array}{cc}
q&s \\ \hline
18 & [1, 2] \\
18 & [1, 4] \\ \hline
18 & [1, 6] \\
18 & [2, 3] \\ \hline
\end{array}
$
\end{minipage}\hfill
\end{table}

\begin{example}\label{ex:P^1(H)q=4}
We study here the first pair in Table~\ref{table:P^1(H)untwisted} and the two first pairs in Table~\ref{table:P^1(H)twisted}.
Let $q=4$, $s=(0,1)$, $s'=(1,2)$.
The corresponding congruence lattices $\mathcal L:=\mathcal L_{q,s}$ and $\mathcal L'=\mathcal L_{q,s'}$ are the same as in Example~\ref{ex:L_4(0,1)1L_4(1,2)1}.
One can easily check that
$$
D_2\cap\mathcal L =\Z(0,4)\oplus\Z(2,0)
\qquad\text{and}\qquad
D_2\cap\mathcal L'=\Z(4,0)\oplus\Z(0,2).
$$
It follows immediately that $D_2\cap\mathcal L$ and $D_2\cap\mathcal L'$ are $\norma{\cdot}_\infty$-isospectral since they differ only by a permutation of coordinates, which preserves $\norma{\cdot}_\infty$.
However, the orbifolds $\Gamma_{q,s}\ba P^1(\Ha)$ and $\Gamma_{q,s'}\ba P^1(\Ha)$ are not isometric by Proposition~\ref{prop:non-isom}.
\end{example}

\begin{remark}\label{rem:P^1(H)situation}
Every example in Table~\ref{table:P^1(H)untwisted} is explained as in the previous example, that is, $D_2\cap \mathcal L_{q,s}$ and $D_2\cap \mathcal L_{q,s'}$ are $\norma{\cdot}_\infty$-isometric.
\end{remark}

\subsection{More remarks}
Here we include more remarks.
For a recent account on spectral theory of orbifolds we refer the reader to \cite{Gordon12-orbifold}.

\begin{remark}
Miatello and Rossetti~\cite{MR02-comparison} studied twisted isospectrality on compact flat manifolds.
It is not difficult to construct a twisted isospectral pair between a compact flat manifold and a compact flat orbifold with singularities like in Remark~\ref{rem:twistediso-mfd-orb}.
Gordon and Rossetti~\cite{GordonRossetti03} gave similar examples replacing twisted spectrum by the middle spectrum.
\end{remark}

\begin{remark}
Rossetti, Schueth and Weilandt~\cite{RossettiSchuethWeilandt08} showed examples of untwisted isospectral orbifolds such that the maximal isotropy order of singular points are different (see also \cite{ShamsStanhopeWebb06}).
Such examples do not appear in this article.
Indeed, for $M=G/K=P^n(\C),P^1(\Ha)$, it is clear that the maximal isotropy order of $\Gamma_{q,s}\ba M$ is $q$ since $\Gamma_{q,s}\subset gKg^{-1}$ for some $g\in G$ (see \cite[Thm.~2.5~(ii)]{RossettiSchuethWeilandt08}).
When $M=S^{2n},S^{2n-1}$, the maximal isotropy order $\Gamma_{q,s}\ba M$ is determined by the set $\{\gcd(s_i,q): 1\leq i\leq n\}$.
One can check that these sets coincide for each untwisted isospectral pair in Table~\ref{table:S^nuntwisted} (cf.\ \cite[\S3.4]{FarsiProctorSeaton14}).
\end{remark}

\subsection{A non-cyclic example}

We set
$\Gamma=\langle\gamma_1,\gamma_2,\gamma_3\rangle$ and $\Gamma'=\langle\gamma_1',\gamma_2',\gamma_3'\rangle$,
where
\begin{equation}
\begin{array}{r@{\;\,=\,\diag(}r@{}r@{}r@{}r@{}r@{}r@{\qquad}r@{\;\,=\,\diag(}r@{}r@{}r@{}r@{}r@{}r}
\gamma_1 & -I_2, & -I_2, &  I_2, &  I_2, &  I_2, &  I_2 ),&    \gamma_1'& -I_2, & -I_2, &  I_2, &  I_2, &  I_2, &  I_2 ),\\
\gamma_2 &  I_2, &  I_2, & -I_2, & -I_2, &  I_2, &  I_2 ),&    \gamma_2'& -I_2, &  I_2, & -I_2, &  I_2, &  I_2, &  I_2 ),\\
\gamma_3 &  I_2, &  I_2, &  I_2, &  I_2, & -I_2, & -I_2 ),&    \gamma_3'& -I_2, & -I_2, & -I_2, & -I_2, & -I_2, & -I_2 ).\\[2mm]
\end{array}
\end{equation}
Here $I_2$ denotes the $2\times2$ identity matrix.
Then $\Gamma$ and $\Gamma'$ are subgroups of the maximal torus in $\SO(12)$ and the orbifolds $\Gamma\ba S^{11}$ and $\Gamma'\ba S^{11}$ are not manifolds.
Indeed, each element in $\Gamma$ and in $\Gamma'$ does not act without fixed points.
One can check that
\begin{align*}
  \mathcal L_\Gamma &=
  \left\{\mu= (a_1,\dots,a_6)\in\Z^6:
  \begin{array}{r}
  a_1+a_2\equiv 0\pmod 2\\
  a_3+a_4\equiv 0\pmod 2\\
  a_5+a_6\equiv 0\pmod 2
  \end{array}
  \right\},\\
  \mathcal L_\Gamma' &=
  \left\{\mu= (a_1,\dots,a_6)\in\Z^6:
  \begin{array}{r}
  a_1+a_2\equiv 0\pmod 2\\
  a_3+a_4\equiv 0\pmod 2\\
  a_1+\dots+a_6\equiv 0\pmod 2
  \end{array}
  \right\}.
\end{align*}

These lattices were already used by Conway and Sloane in \cite{ConwaySloane92} to give lattices in dimension $6$ with the same theta series.
In our words, these lattices are \emph{two-norm isospectral}, it means that for each $t\geq0$ we have that
\begin{equation}\label{eq4:two-norm_isospectral}
\#\{\mu\in\mathcal L_{\Gamma}: \|\mu\|_2=t\}
= \#\{\mu\in\mathcal L_{\Gamma'}: \|\mu\|_2=t\}.
\end{equation}
Here $\|\mu\|_2=\sqrt{a_1^2+\dots+a_n^2}$.
By \cite{Milnor64}, this pair induces a pair of $6$-dimensional flat tori isospectral with respect to the Laplace operator.
Also in \cite{ConwaySloane92}, the authors construct $4$-dimensional two-norm isospectral lattices.

Write $\Psi:\Z^6\to (\Z/2\Z)^6$ given by the class modulo $2$ in each coordinate.
We have that $\mathcal L_{\Gamma}= \Psi^{-1}(C)$ and $\mathcal L_{\Gamma'}= \Psi^{-1}(C')$, where $C$ and $C'$ are the codes in $(\Z/2\Z)^6$ given by
\begin{align}
C:\quad&
\begin{array}{cccccc}
0&0&0&0&0&0\\
1&1&0&0&0&0\\
0&0&1&1&0&0\\
0&0&0&0&1&1\\
1&1&1&1&1&1\\
1&1&1&1&0&0\\
1&1&0&0&1&1\\
0&0&1&1&1&1
\end{array},
\qquad\qquad
C':\quad
\begin{array}{cccccc}
0&0&0&0&0&0\\
1&1&0&0&0&0\\
1&0&1&0&0&0\\
0&1&1&0&0&0\\
1&1&1&1&1&1\\
0&1&0&1&1&1\\
1&0&0&1&1&1\\
0&0&1&1&1&1
\end{array}.
\end{align}
By this clever way to see these lattices, it follows that the lattices are two-norm isospectral since the preimages of each row clearly satisfy \eqref{eq4:two-norm_isospectral}.

It is also clear that the preimages of each row are also one-norm isospectral, then so are $\mathcal L_{\Gamma}$ and $\mathcal L_{\Gamma'}$.
By Theorem~\ref{thm:characterization} we obtain that the orbifolds $\Gamma\ba S^{11}$ and $\Gamma'\ba S^{11}$ are isospectral with respect to the untwisted Laplace operator on functions.
Moreover, similarly as in Corollary~\ref{cor:duality-spherical}, this pair induces a pair of isospectral orbifolds covered by $S^{12}$.

Additionally, one can check also immediately that $\mathcal L_{\Gamma}$ and $\mathcal L_{\Gamma'}$ are $\norma{\cdot}_1^*$-isospectral in the sense of \cite{LMR15-onenorm}.
This implies that the orbifolds $\Gamma\ba S^{11}$ and $\Gamma'\ba S^{11}$ are $p$-isospectral for all $p$ by \cite[Thm.~3.9]{LMR15-onenorm}.
However, contrary to all examples in \cite{LMR15-onenorm}, $\Gamma$ and $\Gamma'$ are representation equivalent in $\SO(12)$, thus the orbifolds are already strongly equivalent.
The fact that $\Gamma$ and $\Gamma'$ are almost conjugate (which implies representation equivalent) follows immediately from
\begin{align*}
\gamma_1&=\gamma_1', &
\gamma_3&\sim \gamma_1'\gamma_2', &
\gamma_1\gamma_3 & \sim \gamma_1'\gamma_2'\gamma_3', &
\gamma_1\gamma_2\gamma_3&=\gamma_3',\\
\gamma_2&\sim\gamma_2', &
\gamma_1\gamma_2 &\sim \gamma_2'\gamma_3', &
\gamma_2\gamma_3 & =\gamma_1'\gamma_3'.
\end{align*}
Here, the symbol $\sim$ means that they are conjugate in $\SO(12)$, or equivalently, they have the same eigenvalues.

Any pair of almost conjugate groups given by $n\times n$ diagonal matrices with entries $\pm1$, induces a pair of strongly isospectral orbifolds covered by $S^{2n-1}$ (and also by $S^{2n}$).
Of course, the orbifolds are isometric if and only if the subgroups are conjugate.
In \cite[\S2]{LMR13-hearcoho}, there are several such examples.
For instance, \cite[Example~4.3]{LMR13-hearcoho} shows a family of eight subgroups of $\SO(24)$ that are almost conjugate (and not conjugate) pairwise.

In conclusion, these kind of examples come from Sunada's method (or representation equivalent method) and not exclusively from the one-norm method studied in \cite{LMR15-onenorm} and in this paper.

\subsection{A more general characterization}
It is clear that the orbifold covered by $S^{2n}$ with cyclic fundamental group of order $q$ are in correspondences with the ones covered by $S^{2n-1}$.
This holds because the maximal torus on $\SO(2n+1)$ and $\SO(2n)$ coincide.
Moreover, we have seen that the characterization of twisted and untwisted isospectral orbifolds with cyclic fundamental group is the same in both cases since the associated congruence lattices coincide.
We conclude this article by giving a more general relation between isospectral orbifolds covered by $S^{2n}$ and $S^{2n-1}$.

\begin{corollary}\label{cor:duality-spherical}
Let $G=\SO(2n+1)$, $H=\{g\in G: ge_{2n+1}=e_{2n+1}\}\simeq\SO(2n)$ and $K=\{g\in G: ge_{2n+1}=e_{2n+1},\;ge_{2n}=e_{2n}\}\simeq\SO(2n-1)$.
Let $\Gamma$ and $\Gamma'$ be finite subgroups of $H$ and let $\chi$ and $\chi'$ be representations of $\Gamma$ and $\Gamma'$ respectively.
Then, $\Delta_{S^{2n},\Gamma,\chi}$ and $\Delta_{S^{2n},\Gamma',\chi'}$ are isospectral if and only if $\Delta_{S^{2n-1},\Gamma,\chi}$ and $\Delta_{S^{2n-1},\Gamma',\chi'}$ are isospectral
\end{corollary}

\begin{proof}
Throughout the proof, $\pi_k$ and $\sigma_k$ denote the elements in $\widehat G_H$ and $\widehat H_K$ respectively with highest weight $k\varepsilon_1$.
By the classical branching law from $G$ to $H$ (see for instance \cite[Thm.~9.16]{Knapp-book-beyond}), we have that $\pi_k|_H\simeq\oplus_{l=0}^k\, \sigma_l$.
Consequently,
\begin{align*}
[\chi:\pi_k|_\Gamma] &= \sum_{l=0}^k \;[\chi:\sigma_l|_\Gamma],\\
[\chi':\pi_k|_{\Gamma'}] &= \sum_{l=0}^k \;[\chi':\sigma_l|_{\Gamma'}],
\end{align*}
for every $k\geq0$.
It follows that $[\chi:\pi_k|_{\Gamma}]=[\chi':\pi_k|_{\Gamma'}]$ for every $k\geq0$ if and only if $[\chi':\sigma_k|_{\Gamma}]=[\chi':\sigma_k|_{\Gamma'}]$ for every $k\geq0$, hence the proof is complete, by Theorem~\ref{thm:spec_chi_rankone}.
\end{proof}

This implies that each pair of $(2n-1)$-dimensional untwisted isospectral spherical space forms (or orbifolds) in \cite{Ikeda80_isosp-lens}, \cite{Ikeda83}, \cite{Gilkey85}, \cite{Wolf01}, \cite{GornetMcGowan06}, \cite{Shams11}, \cite{LMR15-onenorm}, \cite{DeFordDoyle14}, induces a pair of $2n$-dimensional isospectral orbifolds covered by $S^{2n}$.

\begin{remark}
There is a similar result in \cite[page~261]{Pesce96} for $\chi$ trivial.
Here, Pesce describes a method to produce a pair of isospectral orbifolds covered by $P^{2n-1}(\C)$ from a pair of isospectral manifolds covered by $S^{2n-1}$, making use of \cite[Lemme~9]{Pesce96}.
\end{remark}

\bibliographystyle{plain}

\end{document}